\documentclass[12pt,letterpaper]{amsart}

\usepackage{euscript,amsfonts,amssymb,amsmath,amscd,amsthm,enumerate,hyperref}

\usepackage{tikz,bm,float}
\usetikzlibrary{calc}
\colorlet{lgray}{white!85!black}
\colorlet{lred}{white!75!red}
\newcommand{\bra}[1]{\left\langle #1\right|}
\newcommand{\ket}[1]{\left|#1\right\rangle}
\def\b{\bar}

\usepackage{comment}

\usepackage{graphicx}

\usepackage{color}

\usepackage[margin=1.1in]{geometry}
\newtheorem{theorem}{Theorem} 
\newtheorem*{theorem*}{Theorem}
\newtheorem{lemma}[theorem]{Lemma}
\newtheorem{definition}[theorem]{Definition}
\newtheorem{proposition}[theorem]{Proposition}

\newtheorem{corollary}[theorem]{Corollary}

\theoremstyle{remark}
\newtheorem{remark}[theorem]{Remark}

\numberwithin{equation}{section} \numberwithin{theorem}{section}

\newcommand{\la}{\lambda}

\newcommand{\ii}{{\mathbf i}}

\renewcommand{\c}{\mathbf c}

\newcommand{\eps}{\varepsilon}

\usepackage{MnSymbol}

\usepackage{skak}

\usepackage{adjustbox}
\usetikzlibrary{arrows}
\usetikzlibrary{decorations.pathreplacing}

\newcommand{\all}{\mathbf{+}}
\newcommand{\nul}{\zugzwang}
\newcommand{\rr}{\mathbf{--}}
\newcommand{\uu}{\mathbf{|}}
\newcommand{\ur}{\mathbf{ \lefthalfcap }}
\newcommand{\ru}{\mathbf{ \righthalfcup }}

\def\ie{{\it i.e.}\/\ }

\def\cf{{\it cf.}\/\ }
\def\sq{sqrt(3)}

\newcounter{x}
\newcounter{y}
\newcounter{z}
\newcommand\xaxis{210}
\newcommand\yaxis{-30}
\newcommand\zaxis{90}
\newcommand\topside[3]{
  \fill[fill=black!40!white, draw=black,shift={(\xaxis:#1)},shift={(\yaxis:#2)},
  shift={(\zaxis:#3)}] (0,0) -- (30:1) -- (0,1) --(150:1)--(0,0);
}
\newcommand\leftside[3]{
  \fill[fill=black!10!white, draw=black,shift={(\xaxis:#1)},shift={(\yaxis:#2)},
  shift={(\zaxis:#3)}] (0,0) -- (0,-1) -- (210:1) --(150:1)--(0,0);
}
\newcommand\rightside[3]{
  \fill[fill=black!25!white, draw=black,shift={(\xaxis:#1)},shift={(\yaxis:#2)},
  shift={(\zaxis:#3)}] (0,0) -- (30:1) -- (-30:1) --(0,-1)--(0,0);
}
\newcommand\cube[3]{
  \topside{#1}{#2}{#3} \leftside{#1}{#2}{#3} \rightside{#1}{#2}{#3}
}
\newcommand\planepartition[1]{
 \setcounter{x}{-1}
  \foreach \a in {#1} {
    \addtocounter{x}{1}
    \setcounter{y}{-1}
    \foreach \b in \a {
      \addtocounter{y}{1}
      \setcounter{z}{-1}
      \foreach \c in {0,...,\b} {
        \addtocounter{z}{1}
      \ifthenelse{\c=0}{\setcounter{z}{-1},\addtocounter{y}{0}}{
        \cube{\value{x}}{\value{y}}{\value{z}}}
      }
    }
  }
}

\title[Between the stochastic six vertex model and Hall--Littlewood processes]
{Between the stochastic six vertex model and Hall--Littlewood processes}

\author{Alexei Borodin}

\address[Alexei Borodin]{ Department of Mathematics, MIT, Cambridge, USA, and
Institute for Information Transmission Problems, Moscow, Russia. E-mail: borodin@math.mit.edu }

\author{Alexey Bufetov}

\address[Alexey Bufetov]{Department of Mathematics, MIT, Cambridge, USA. E-mail: alexey.bufetov@gmail.com}

\author{Michael Wheeler}

\address[Michael Wheeler]{ School of Mathematics and Statistics, University of Melbourne, Parkville,
Victoria 3010, Australia. E-mail: wheelerm@unimelb.edu.au}

\begin{document}

\maketitle

\begin{abstract}
We prove that the joint distribution of the values of the height function for the stochastic six vertex model in a quadrant along a down-right path coincides with that for the lengths of the first columns of partitions distributed according to certain Hall--Littlewood processes. 

In the limit when one of the quadrant axes becomes continuous, we also show that the two-dimensional random field of the height function values has the same distribution as the lengths of the first columns of partitions from certain ascending Hall--Littlewood processes evolving under a Robinson--Schensted--Knuth type Markovian evolution. 

\end{abstract}

\setcounter{tocdepth}{1}
\tableofcontents

\section{Introduction}

The last two decades have seen a sharp increase in the number of \emph{integrable}, or exactly solvable probabilistic systems. Among others, two fairly general algebraic mechanisms of producing such systems were suggested --- the Macdonald processes \cite{BC}, \cite{BP-lectures}, and the higher spin stochastic six vertex models \cite{CP}, \cite{BP1}, \cite{BP-hom}.

While these two mechanisms seemed rather different at first, the distance between them appears to be shrinking. One direct connection was noticed in \cite{B-6v}, where a distributional equality between the height function at a vertex for a stochastic six vertex in a quadrant, and the (shifted) length of the first column for the Young diagram distributed according to a Hall--Littlewood measure\footnote{The Hall--Littlewood measures are probability distributions on partitions that can be viewed as marginals of the Hall--Littlewood processes -- a specialization of the Macdonald processes at $q=0$, where $(q,t)$ are the parameters of the underlying Macdonald polynomials.} was established. The goal of the present paper is to extend this equality to multiple observation points. Another intriguing connection between the stochastic vertex models and $q$-Whittaker ({\it i.e.}, $t=0$ Macdonald) processes was developed in \cite{OP} very recently; at the moment this does not seem to be immediately related to what we do below. 

The one-point observation of \cite{B-6v} was based on explicit integral representations of certain exponential moments of both distributions that had been obtained in previous works; those turned out to be apparently the same, once the parameters of the models were matched appropriately. We first extend this comparison to the joint distributions of the height function for the stochastic six vertex model taken at vertices that are located along a horizontal or a vertical line. The corresponding exponential moments of the height function were computed in \cite{BP1}, and they match those for the (shifted) first columns of partitions coming from \emph{ascending} Hall--Littlewood processes (a subclass of Hall--Littlewood processes supported by ascending chains of partitions). This match is sufficient to claim the equality of multi-point distributions, and we give the details in Section \ref{sc:moments-match}. 

Unfortunately, such a straightforward approach fails if one wants to consider the values of the six vertex height function at vertices that do not lie on a straight line, as no explicit formulas describing joint distributions of such values are known. 

Nevertheless, we are able to prove a more general distributional match between the six vertex height function along an arbitrary down-right lattice path in the quadrant on one side, and the first columns of partitions distributed according to certain general (not necessarily ascending) Hall--Littlewood processes. To do that we employ a different approach based on an infinite volume limit of the algebraic Bethe ansatz for the quantum affine $sl_2$ and its limiting algebra of $t$-bosons, as (independently) developed in \cite{B} and \cite{WZ}, see also \cite{BP1}, \cite{BP-hom}. 

This establishes a link between many marginals of the stochastic six vertex in the quadrant and Hall--Littlewood processes, but leaves the question of the Hall--Littlewood interpretation of the whole quadrant open. 

In the last section of the present work, we address this question in a degenerate case of the stochastic six vertex model when one of the boundary axes of the quadrant becomes continuous. If one looks at the vector of height function values of the stochastic six vertex model along a vertical section of the quadrant (it suffices to take finitely many bottom-most values), then, as the section moves to the right, this vector evolves in a Markovian way. The continuous axis limit corresponds to the situation when this Markov chain evolves in continuous time, and thus makes at most one elementary jump at any time moment.

As follows from our previous results, at each time moment, the distribution of the vector is that of first columns for an ascending Hall--Littlewood process. We show that the so-called RSK (Robinson--Schensted--Knuth) Markovian evolution of the ascending Hall--Littlewood processes, as constructed in \cite{BufP}, \cite{BP2}, projects to a Markov chain on the first columns, and this projection coincides with the Markov chain originating from the stochastic six vertex model. 

Thus, in the limit of one continuous axis, we show that the whole random field of the values of the height function for the stochastic six vertex in the quadrant can be seen as the first column marginal for the RSK dynamics on the ascending Hall--Littlewood processes. This, in particular, also implies the result about the height function along down-right paths, but only in the continuous axis degeneration. 

An obvious question is whether this claim can be extended to the original two-dimensional lattice, which should correspond to discrete time RSK dynamics. The answer to this question appears to be positive, and the needed dynamics will be described in an upcoming work \cite{BM}. 

Another obvious question is whether the extension of the stochastic six vertex obtained by looking at further columns of the corresponding Hall--Littlewood-distributed partitions can be given an independent interpretation in terms of some solutions of the Yang--Baxter equation. This remains unclear to us at the moment. 

\subsection{Acknowledgments} A.~B. was partially supported by the NSF grant DMS-1607901 and by  Fellowships of the Radcliffe Institute for Advanced Study and the Simons Foundation. M.~W. is supported by the Australian Research Council grant DE160100958.

\section{Hall--Littlewood processes}

\subsection{Definitions related to partitions}
\label{sec:partitions}

A \textit{partition} $\la$ is a finite non-increasing sequence of positive integers $\la_1 \ge \la_2 \ge \cdots$. The \textit{length} of a partition $\la$ is the number of positive integers $\la_i$ that constitute it, and equal to $\la'_1$, the first part of the \textit{conjugate partition} $\la'$; more generally, $\la'_j = \#\{i:\la_i \geq j\}$. For two partitions $\la, \mu$ we write $\la \subset \mu$ if the inequalities $\la_1 \le \mu_1$, $\la_2 \le \mu_2, \dots $ hold. Similarly, we write $\lambda \prec \mu$ and say that $\la$ and $\mu$ interlace if the string of inequalities $\mu_1 \ge \la_1 \ge \mu_2 \ge \la_2 \geq \mu_3 \geq \cdots$ hold.

Let $\la$ be a partition represented by its Young diagram, \ie by left-justified rows of boxes of length $\la_i$, for $1 \leq i \leq \la'_1$. The {\it frame} of $\la$ is the union of all eastward and southward facing edges of the Young diagram. For example, for $\la = (6,3,3,1)$, we have the Young diagram
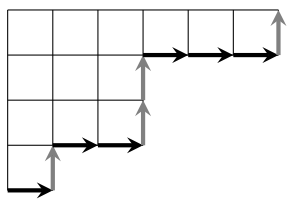
\begin{align*}
\begin{tikzpicture}[scale=0.6,>=stealth]
\draw (0,4) -- (6,4);
\draw (0,3) -- (6,3);
\draw (0,2) -- (3,2);
\draw (0,1) -- (3,1);
\draw (0,0) -- (1,0);
\draw (0,0) -- (0,4);
\draw (1,0) -- (1,4);
\draw (2,1) -- (2,4);
\draw (3,1) -- (3,4);
\draw (4,3) -- (4,4);
\draw (5,3) -- (5,4);
\draw (6,3) -- (6,4);
\draw[ultra thick,->] (0,0) -- (1,0);
\draw[ultra thick,gray,->] (1,0) -- (1,1);
\draw[ultra thick,->] (1,1) -- (2,1);
\draw[ultra thick,->] (2,1) -- (3,1);
\draw[ultra thick,gray,->] (3,1) -- (3,2);
\draw[ultra thick,gray,->] (3,2) -- (3,3);
\draw[ultra thick,->] (3,3) -- (4,3);
\draw[ultra thick,->] (4,3) -- (5,3);
\draw[ultra thick,->] (5,3) -- (6,3);
\draw[ultra thick,gray,->] (6,3) -- (6,4);
\end{tikzpicture}
\end{align*}
where we have marked the frame of $\lambda$ in black and grey steps.

Given a binary string $S$ of $p$ pluses and $m$ minuses, we can associate to it a partition $\lambda(S)$: reading the signs in $S$, we trace a path consisting of $p$ up steps and $m$ right steps, where $+$ corresponds with an up step and $-$ with a right step. Embedding this path in a $p \times m$ box, it then frames the Young diagram of the partition $\lambda(S)$. In the example above, the corresponding binary string is $S = (-,+,-,-,+,+,-,-,-,+)$.

\subsection{Ascending Hall--Littlewood process}
\label{sec:HL-process-usual}

In this section we briefly recall the definition of a particular case of the Hall--Littlewood process and compute expectations of its observables. We refer to \cite[Chapter 3]{M} for the definition and properties of the Hall--Littlewood symmetric functions that are used in this construction, and to \cite[Section 2]{BC} for a more general definition of the Macdonald processes.


Let $P_{\la}$, $Q_{\la}$ be the Hall--Littlewood symmetric functions that depend on a real parameter $t$, $0<t<1$. For a positive integer $M$ let $\rho$ be a Macdonald-positive specialization of the algebra of symmetric functions (see the definition in, {\it e.g.}, \cite[Section 2]{BC}), and let $\{ a_i \}_{i=1}^{\infty}$ be positive reals.

For $M$ a positive integer, let $\la^{(1)} \subset \la^{(2)} \subset \dots \subset \la^{(M)}$ be a sequence of partitions. We define a weight of this sequence as
\begin{multline}
W_{\{ a_i \},\rho} ( \la^{(1)} \subset \la^{(2)} \subset \dots \subset \la^{(M)} ) :=
\\
P_{\la^{(1)}} (a_1) P_{\la^{(2)} / \la^{(1)}} (a_2) \dots P_{\la^{(M)} / \la^{(M-1)}} (a_M) Q_{\la^{(M)}} (\rho).
\end{multline}
Next, it is known (see \cite[Section 2]{BC}) that under certain conditions on $\{ a_i \}_{i\le M}$ and $\rho$ one can define a probability measure on such sequences of partitions by
\begin{equation}
\label{eq:HL-proc-def-1/5}
\mathrm{Prob}_{ \{ a_i \},\rho} (\la^{(1)} \subset \la^{(2)} \subset \dots \subset \la^{(M)}) = \frac{W_{\{ a_i \},\rho} ( \la^{(1)} \subset \la^{(2)} \subset \dots \subset \la^{(M)} )}{ C(a_1,\dots, a_M;\rho)},
\end{equation}
where $C(a_1,\dots, a_M;\rho)$ is a normalization constant; its explicit form will not be important for us. Following the terminology of \cite{BC}, the process described by the measure \eqref{eq:HL-proc-def-1/5} can be called ``ascending'', since the sequence $\la^{(1)} \subset \la^{(2)} \subset \dots \subset \la^{(M)}$ is non-decreasing.

We will need two specific choices of a specialization $\rho$. Let $\{b_j \}_{j=1}^{\infty}$ be real numbers which satisfy
\begin{equation}
\label{eq:cond-HL}
a_i >0, \qquad b_j >0, \qquad a_i b_j <1, \qquad \mbox{for all $i$ and $j$.}
\end{equation}
For a positive integer $N$ let us define the specialization $\rho_N^-$ by setting the variables $x_1,\dots, x_N$ (from the algebra of symmetric functions) to be equal to $b_1, \dots, b_N$, respectively. In this case, \eqref{eq:HL-proc-def-1/5} reads
\begin{multline}
\label{eq:HL-proc-def-1}
\mathrm{Prob}_{M,N} (\la^{(1)} \subset \la^{(2)} \subset \dots \subset \la^{(M)}) \\ :=  \frac{P_{\la^{(1)}} (a_1) P_{\la^{(2)} / \la^{(1)}} (a_2) \dots P_{\la^{(M)} / \la^{(M-1)}} (a_M) Q_{\la^{(M)}} (b_1, \dots, b_N)}{\prod_{j=1}^N \prod_{i=1}^M \frac{1 - t a_i b_j}{1 - a_i b_j}}.
\end{multline}

We will also need another ascending Hall--Littlewood process with $\rho$ being the Plancherel specialization with parameter $\tau$, for $\tau>0$ (see \cite[Definition 2.2.3]{BC} for a definition of the Plancherel specializations).  We will use this process in Section \ref{sec:RSKdyn}.


The rest of the section is devoted to random partitions $\la^{(1)} \subset \la^{(2)} \subset \dots \subset \la^{(M)}$ which are distributed according to \eqref{eq:HL-proc-def-1}. For each $1 \le m \le M$, the marginal distribution of the partition $\la^{(m)}$ is the Hall--Littlewood measure with parameters $a_1, \dots, a_m$ and $b_1, \dots, b_N$ (which means that the probability of the event $\la^{(m)} = \la$ is proportional to $P_{\la} (a_1, \dots, a_m) Q_{\la} (b_1, \dots, b_N)$), \cf \cite[Section 2.2]{BC}.

We denote by $\lambda'_1 (m,N) \equiv \lambda'_1(m)$ the length of the first column of the random partition $\la^{(m)}$. The machinery of Macdonald processes allows to compute certain functionals of such random variables. We will need the following technical definition.


%

\begin{definition}
\label{def:contours-HL}
For positive integers $N,k$, and $m_1 \ge m_2 \ge \dots \ge m_k \ge 1$, a collection of curves $\gamma_1$, $\dots$, $\gamma_k$ in the complex plane is called \textit{suitable for} $(t, \{a_i \}_{i=1}^{m_1}, \{b_j \}_{j=1}^N)$, if the following conditions hold:
\begin{enumerate}[{\bf 1.}]

\item For any $1 \le r \le k$ the curve $\gamma_r$ is a union of finitely many simple positively oriented smooth closed curves.

\item For any $1 \le r \le k$ the curve $\gamma_r$ encircles $0, a_1, \dots, a_{m_r}$, and does not encircle any of the points $\{ 1/(t b_j) \}_{j=1}^N$.

\item For any $1 \le r <s \le k$ the interior of the curve $\gamma_s$ contains no points from the curve $t \gamma_r$ (which is the image of $\gamma_r$ under $z \mapsto t z$).

\end{enumerate}

\end{definition}

An example of an arrangement of such parameters and contours is shown in Figure \ref{fig:contoursHLand6vert} (left panel).
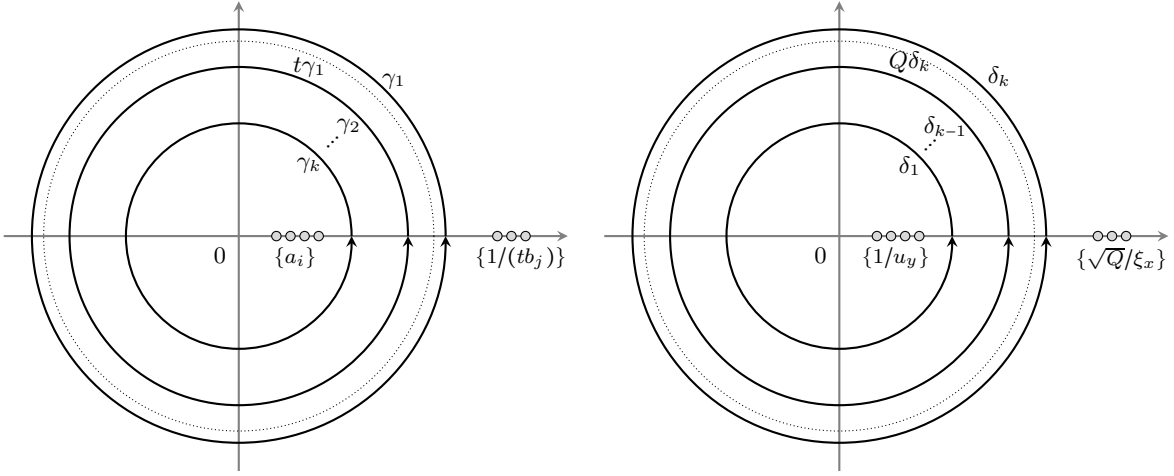
\begin{figure}
\begin{tabular}{cc}
\begin{tikzpicture}[xscale=1.25,yscale=1.25,>=stealth]
\draw[gray,thick,->] (-2.5,0) -- (3.5,0);
\draw[gray,thick,->] (0,-2.5) -- (0,2.5);
\node[below,text centered] at (0.6,0) {\tiny$\{a_i\}$};
\node[below,text centered] at (3,0) {\tiny$\{1/(t b_j)\}$};
\foreach\x in {0.45,0.6,0.75,0.9,2.8,2.95,3.1}{
\filldraw[fill=lgray] (\x,0) arc (0:-360:0.05);
}
\draw[thick,<-] (1.2,0) arc (0:-360:1.2);
\draw[thick,<-] (1.8,0) arc (0:-360:1.8);
\draw[thick,<-] (2.2,0) arc (0:-360:2.2);
\draw[densely dotted] (2.075,0) arc (0:-360:2.075);
\node at (-0.2,-0.2) {\scriptsize $0$};
\node at (1.65,1.65) {\scriptsize$\gamma_1$};
\node at (0.75,1.8) {\scriptsize$t\gamma_1$};
\node at (1.17,1.17) {\scriptsize$\gamma_2$};
\node[rotate=-45] at (1,1) {\tiny$\vdots$};
\node at (0.75,0.75) {\scriptsize$\gamma_k$};
\end{tikzpicture}
&
\begin{tikzpicture}[xscale=1.25,yscale=1.25,>=stealth]
\draw[gray,thick,->] (-2.5,0) -- (3.5,0);
\draw[gray,thick,->] (0,-2.5) -- (0,2.5);
\node[below,text centered] at (0.6,0) {\tiny$\{1/u_y\}$};
\node[below,text centered] at (3,0) {\tiny$\{\sqrt{Q}/\xi_x\}$};
\foreach\x in {0.45,0.6,0.75,0.9,2.8,2.95,3.1}{
\filldraw[fill=lgray] (\x,0) arc (0:-360:0.05);
}
\draw[thick,<-] (1.2,0) arc (0:-360:1.2);
\draw[thick,<-] (1.8,0) arc (0:-360:1.8);
\draw[thick,<-] (2.2,0) arc (0:-360:2.2);
\draw[densely dotted] (2.075,0) arc (0:-360:2.075);
\node at (-0.2,-0.2) {\scriptsize $0$};
\node at (1.7,1.7) {\scriptsize$\delta_k$};
\node at (0.75,1.85) {\scriptsize$Q\delta_k$};
\node at (1.13,1.17) {\scriptsize$\delta_{k-1}$};
\node[rotate=-45] at (0.97,0.97) {\tiny$\vdots$};
\node at (0.75,0.75) {\scriptsize$\delta_1$};
\end{tikzpicture}
\end{tabular}
\caption{ A possible arrangement of parameters and contours for Definitions \ref{def:contours-HL} (left panel) and \ref{def:VERTmom} (right panel). With matching \eqref{eq:param-match} and a change of variables $z=1/w$, these contours correspond to each other.}
\label{fig:contoursHLand6vert}
\end{figure}


\begin{proposition}
\label{prop:HL-mom}

Let $N,\,k$, and $m_1 \ge m_2 \ge \dots \ge m_k$ be positive integers. Let $\lambda'_1 (m_1,N),\dots,\lambda'_1 (m_k,N)$ be random variables constructed as above with the use of the ascending Hall--Littlewood process $\mathrm{Prob}_{m_1,N}$.
Assume that contours $\gamma_1, \dots, \gamma_k$ are suitable for $(t, \{a_i \}_{i=1}^{m_1}, \{b_j \}_{j=1}^N)$. We have

\begin{multline}
\label{eq:HLmom}
\mathbf E \left( t^{m_1 - \lambda'_1 (m_1) } t^{m_2 - \lambda'_1 (m_2) } \dots t^{m_k - \lambda'_1 (m_k) } \right)  \\ = \frac{t^{k(k-1)/2}}{(2 \pi \ii)^{k} } \oint_{\gamma_1} \frac{d z_1}{z_1} \oint_{\gamma_2} \frac{d z_2}{z_2} \dots \oint_{\gamma_k} \frac{d z_k}{z_k} \prod_{1 \le i<j \le k} \frac{z_i - z_j}{t z_i - z_j} \prod_{l=1}^k \prod_{j=1}^N \frac{1 - z_l b_j}{1 - t z_l b_j} \prod_{l=1}^k \prod_{i=1}^{m_l} \frac{t z_l - a_{i}}{z_l - a_{i}}.
\end{multline}
\end{proposition}

\begin{proof}
This is essentially Proposition 2.2.14 of \cite{BC}, see also \cite[Proposition 3.4]{D}. The conditions on contours appear from the residues which we need (and need not) to take into account during the procedure described in \cite[Proposition 2.2.14]{BC}.
\end{proof}

\subsection{More general Hall--Littlewood process}
\label{sec:general-HL}

We will also study a Hall--Littlewood process which allows for arbitrary (not necessarily ascending) sequences of partitions. For positive integers $M,N$, define the two sets
\begin{align*}
\mathcal{S}^{\pm}_{M,N}
=
\left\{
(S(1),\dots,S(M+N))
\Bigg|
\begin{array}{ll} S(1) = \pm 1, & S(i) \in \{-1,+1\} \ \forall\ 1 < i < M+N, \vspace{0.25cm} \\
S(M+N) = \mp 1, & \sum_{i=1}^{M+N} S(i) = M-N
\end{array}
\right\}.
\end{align*}
In words, $\mathcal{S}^{\pm}_{M,N}$ is the set of all binary strings of $M$ pluses and $N$ minuses, which begin with $\pm$ and end with $\mp$. Since we are only concerned with the signs in a binary string $S$, we will frequently abbreviate $+1$ and $-1$ by $+$ and $-$.

Let us choose any such binary string in the first set, $S \in \mathcal{S}^{+}_{M,N}$. Let $p(i)$ be the number of pluses in the substring $(S(1), S(2), \dots, S(i))$, and let $m(i)$ be the number of minuses in the same substring.
Consider a collection of partitions $\la^{(1)} * \la^{(2)} * \dots * \la^{(M+N-1)}$, where $*$ stands for either $\subset$ or $\supset$ in the following way: if $S(i)=+$, then we have $\la^{(i-1)} \subset \la^{(i)}$; if $S(i)=-$, then $\la^{(i-1)} \supset \la^{(i)}$.

Let $\rho_1, \dots, \rho_N$ be Macdonald-positive specializations of the algebra of symmetric functions.
Set
\begin{equation}
\label{eq:HL-proc-def-2/55}
W^{(S,i)}_{ \{ a_j \}, \{ \rho_j \} } :=
\begin{cases}
P_{\la^{(i)} / \la^{(i-1)}} (a_{p(i)}), \qquad & \mbox{if $S(i)=+$}, \\
Q_{\la^{(i-1)} / \la^{(i)}} (\rho_{N-m(i)+1}), \qquad & \mbox{if $S(i)=-$},
\end{cases}
\end{equation}
for all $1 \leq i \leq M+N$, where by agreement $\la^{(0)} = \la^{(M+N)} = \varnothing$, the empty partition.
For certain conditions on $\{ a_j \}$ and $\{ \rho_j \}$ one can define a probability measure on collections of partitions $\la^{(1)}* \la^{(2)} \dots * \la^{(M+N-1)}$ by the formula
\begin{equation}
\label{eq:HL-proc-def-2/5}
\mathrm{Prob}_{ \{ a_j \}, \{ \rho_j \} }^S ( \la^{(1)} * \la^{(2)} * \dots * \la^{(M+N-1)} ) := \frac{\displaystyle \prod_{i=1}^{M+N} W^{(S,i)}_{\{a_j\},\{\rho_j\} }}{C( \{a_j \}; \{ \rho_j \}) },
\end{equation}
where $C( \{a_j \}; \{ \rho_j \})$ is a normalization constant; its explicit form will not be important for us. As in Section \ref{sec:HL-process-usual}, we will need two specific examples of this construction. First, for all $1 \le i \le N$ let us take the specialization $\rho_i$ to be the one variable evaluation $x_1 \mapsto b_i$ of the algebra of symmetric functions (recall that parameters $\{ b_i \}$ satisfy \eqref{eq:cond-HL}). Then \eqref{eq:HL-proc-def-2/55} reads
$$
W^{(S,i)}_{ M, N } :=
\begin{cases}
P_{\la^{(i)} / \la^{(i-1)}} (a_{p(i)}), \qquad & \mbox{if $S(i)=+$}, \\
Q_{\la^{(i-1)} / \la^{(i)}} (b_{N-m(i)+1}), \qquad & \mbox{if $S(i)=-$},
\end{cases}
$$
and we obtain the measure
\begin{equation}
\label{eq:HL-proc-def-2}
\mathrm{Prob}_{M,N}^S ( \la^{(1)} * \la^{(2)} * \dots * \la^{(M+N-1)} ) := \frac{\displaystyle \prod_{i=1}^{M+N} W^{(S,i)}_{M,N}}{\Pi^{S}(a_1,\dots,a_M; b_1,\dots,b_N) },
\end{equation}
where the normalization is given by
\begin{align*}
\Pi^{S}(a_1,\dots,a_M; b_1,\dots,b_N)
=
\prod_{\substack{1 \leq i<j \leq M+N \\ (S(i),S(j)) =(+,-)}}
 \frac{1 - t a_{p(i)} b_{N-m(j)+1}}{1 - a_{p(i)} b_{N-m(j)+1}}.
\end{align*}
We refer again to \cite[Section 2]{BC} for the proof of correctness of such a definition for any $\{a_i \}$ and $\{b_j\}$ which satisfy \eqref{eq:cond-HL}.

It can be straightforwardly checked that for a string $S=(+,\dots,+,-,\dots,-)$, the distribution of $(\lambda^{(1)}, \dots, \la^{(M)})$ coincides with $\mathrm{Prob}_{M,N}$ given by \eqref{eq:HL-proc-def-1}. Also, it can be checked that the marginal distribution of $\lambda^{(i)}$ is the Hall--Littlewood measure with parameters $a_1, \dots, a_{p(i)}$ and $b_1, \dots, b_{N-m(i)}$; see \cite[Section 2.2]{BC}.

We denote by $\lambda'_1 (m,N,S)$ the length of the first column of a random partition $\la^{(m)}$ coming from the process \eqref{eq:HL-proc-def-2}.

For another choice of specializations, let us fix reals $0\le \tau_1 \le \tau_2 \le \dots \le \tau_N$, and let $\rho_1$ be the Plancherel specialization with the parameter $\tau_2 - \tau_1$, ..., $\rho_{N-1}$ be the Plancherel specialization with the parameter $\tau_{N} - \tau_{N-1}$, and let $\rho_N$ be the Plancherel specialization with the parameter $\tau_N$. We will use the process \eqref{eq:HL-proc-def-2/5} with these specializations in Section \ref{sec:RSKdyn}. 


\subsection{A convenient random variable}
\label{sec:support}

It is helpful to define a further random variable coming from the process \eqref{eq:HL-proc-def-2}. We call it the \textit{support} of the sequence $\la^{(1)}* \dots * \la^{(M+N-1)}$, and denote it by $[\la^{(1)}* \dots * \la^{(M+N-1)}]$. It is a skew Young diagram $\nu/\mu$ formed by two partitions $\mu \subset \nu$, where $\mu$ and $\nu$ are obtained in the following way:
\begin{enumerate}[{\bf 1.}]
\item $\mu$ is the partition obtained from the binary string $S$, via the identification explained in Section \ref{sec:partitions}; \ie $\mu \equiv \mu(S)$. Clearly $\mu$ is not random at all, since it is already specified with the choice of measure \eqref{eq:HL-proc-def-2}.

\item Construct from $\la^{(1)}* \dots * \la^{(M+N-1)}$ a second binary string
\begin{align*}
T = (T(1),\dots,T(M+N)) \in \mathcal{S}_{M,N}^{-}.
\end{align*}
For all $1 \leq i \leq M+N$, we take $T(i) = +$ if $S(i) = +$ and $\lambda'_1(i-1) = \lambda'_1(i)$, or if $S(i) = -$ and $\lambda'_1(i-1) = \lambda'_1(i) + 1$. Conversely, we take $T(i) = -$ if $S(i) = +$ and $\lambda'_1(i-1) = \lambda'_1(i) - 1$, or if $S(i) = -$ and $\lambda'_1(i-1) = \lambda'_1(i)$. From this, define the second partition $\nu \equiv \nu(T)$.
\end{enumerate}
The support of the sequence $\la^{(1)}* \dots * \la^{(M+N-1)}$ has a natural interpretation on plane partitions. Any two neighbouring partitions in this sequence are required to interlace (otherwise the measure \eqref{eq:HL-proc-def-2} vanishes), meaning that we can interpret $\lambda^{(i)}$ as the $i$-th diagonal slice in a skew plane partition $\pi$. $[\la^{(1)}* \dots * \la^{(M+N-1)}]$ is then the skew Young diagram on which
$\pi$ is supported: $\mu(S)$ traces the ``inner'' partition, $\nu(T)$ traces the ``outer'' partition, and the plane partition is valued on $\nu/\mu$. In Figures \ref{fig:pp1} and \ref{fig:pp2} we give two examples; Figure \ref{fig:pp1} illustrates support in the case of ascending Hall--Littlewood processes, while Figure \ref{fig:pp2} is an example of the fully generic processes of Section \ref{sec:general-HL}.

\begin{figure}
\begin{tabular}{cc}
\begin{tikzpicture}[scale=0.7,baseline=(current bounding box.center),>=stealth]
\planepartition{{3,3,3,3,2,2},{3,3,1},{3,2,1},{1}}
\draw[ultra thick,->] ({-2*\sq},-2) -- node[midway,below] {\tiny $1$} ({-1.5*\sq},-2.5);
\draw[ultra thick,gray,->] ({-1.5*\sq},-2.5) -- node[midway,below] {\tiny $2$} ({-\sq},-2);
\draw[ultra thick,->] ({-\sq},-2) -- node[midway,below] {\tiny $3$} ({-0.5*\sq},-2.5);
\draw[ultra thick,->] ({-0.5*\sq},-2.5) -- node[midway,below] {\tiny $4$} (0,-3);
\draw[ultra thick,gray,->] (0,-3) -- node[midway,below] {\tiny $5$} ({0.5*\sq},-2.5);
\draw[ultra thick,gray,->] ({0.5*\sq},-2.5) -- node[midway,below] {\tiny $6$} ({\sq},-2);
\draw[ultra thick,->] ({\sq},-2) -- node[midway,below] {\tiny $7$} ({1.5*\sq},-2.5);
\draw[ultra thick,->] ({1.5*\sq},-2.5) -- node[midway,below] {\tiny $8$} ({2*\sq},-3);
\draw[ultra thick,->] ({2*\sq},-3) -- node[midway,below] {\tiny $9$} ({2.5*\sq},-3.5);
\draw[ultra thick,gray,->] ({2.5*\sq},-3.5) -- node[midway,below] {\tiny $10$} ({3*\sq},-3);
\end{tikzpicture}
\quad
&
\quad
\begin{tikzpicture}[scale=0.7,baseline=(current bounding box.center),>=stealth,rotate=-45]
\draw (0,4) -- (6,4);
\draw (0,3) -- (6,3);
\draw (0,2) -- (3,2);
\draw (0,1) -- (3,1);
\draw (0,0) -- (1,0);
\draw (0,0) -- (0,4);
\draw (1,0) -- (1,4);
\draw (2,1) -- (2,4);
\draw (3,1) -- (3,4);
\draw (4,3) -- (4,4);
\draw (5,3) -- (5,4);
\draw (6,3) -- (6,4);
\draw[ultra thick,->] (0,0) -- (1,0);
\draw[ultra thick,gray,->] (1,0) -- (1,1);
\draw[ultra thick,->] (1,1) -- (2,1);
\draw[ultra thick,->] (2,1) -- (3,1);
\draw[ultra thick,gray,->] (3,1) -- (3,2);
\draw[ultra thick,gray,->] (3,2) -- (3,3);
\draw[ultra thick,->] (3,3) -- (4,3);
\draw[ultra thick,->] (4,3) -- (5,3);
\draw[ultra thick,->] (5,3) -- (6,3);
\draw[ultra thick,gray,->] (6,3) -- (6,4);
\node[text centered] at (0.5,3.5) {$3$};
\node[text centered] at (1.5,3.5) {$3$};
\node[text centered] at (2.5,3.5) {$3$};
\node[text centered] at (3.5,3.5) {$3$};
\node[text centered] at (0.5,2.5) {$3$};
\node[text centered] at (1.5,2.5) {$3$};
\node[text centered] at (0.5,1.5) {$3$};
\node[text centered] at (1.5,1.5) {$2$};
\node[text centered] at (4.5,3.5) {$2$};
\node[text centered] at (5.5,3.5) {$2$};
\node[text centered] at (2.5,2.5) {$1$};
\node[text centered] at (2.5,1.5) {$1$};
\node[text centered] at (0.5,0.5) {$1$};
\foreach\x in {1,...,4}{
\draw[dotted] (0,\x) -- (-1,\x+1);
\node at (-1.2,\x+1.2) {\tiny $\lambda^{(\x)}$};
}
\foreach\x in {5,...,9}{
\draw[dotted] (\x-4,4) -- (\x-5,5);
\node at (\x-5.2,5.2) {\tiny $\lambda^{(\x)}$};
}
\end{tikzpicture}
\end{tabular}
\caption{Three-dimensional representation of a plane partition and corresponding two-dimensional projection. In this case $M=4$, $N=6$, $S = (+,+,+,+,-,-,-,-,-,-)$, and the sequence $\lambda^{(1)} \subset \cdots \subset \lambda^{(4)} \supset \cdots \supset \lambda^{(9)}$ can be read as the diagonal slices of the plane partition. The steps numbered $1$ through $10$ determine that $T = (-,+,-,-,+,+,-,-,-,+)$. The support of the sequence is just the Young diagram on which the plane partition is supported; we see that $[\lambda^{(1)} \subset \cdots \subset \lambda^{(4)} \supset \cdots \supset \lambda^{(9)}] = \nu(T)/\mu(S) = (6,3,3,1) / \varnothing$.}
\label{fig:pp1}
\end{figure}
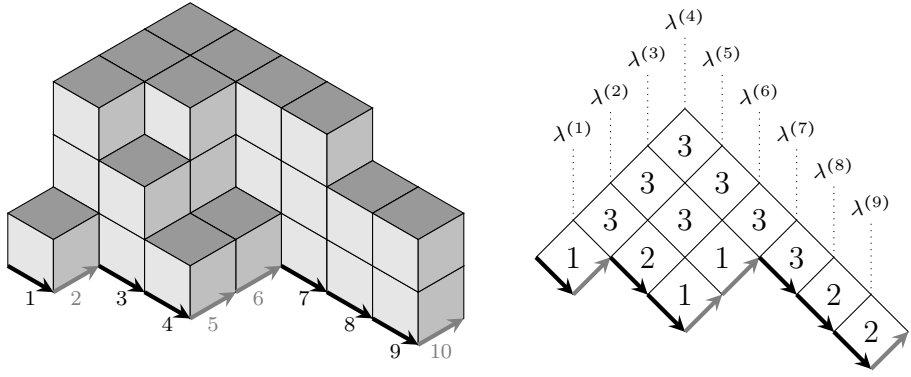

\begin{figure}
\begin{tabular}{cc}
\begin{tikzpicture}[scale=0.7,baseline=(current bounding box.center),>=stealth]
\planepartition{{0,0,3,2,1,1},{0,2,1},{2,2,1},{1}}
\draw[ultra thick,->] ({-2*\sq},-2) -- node[midway,below] {\tiny $1$} ({-1.5*\sq},-2.5);
\draw[ultra thick,gray,->] ({-1.5*\sq},-2.5) -- node[midway,below] {\tiny $2$} ({-\sq},-2);
\draw[ultra thick,->] ({-\sq},-2) -- node[midway,below] {\tiny $3$} ({-0.5*\sq},-2.5);
\draw[ultra thick,->] ({-0.5*\sq},-2.5) -- node[midway,below] {\tiny $4$} (0,-3);
\draw[ultra thick,gray,->] (0,-3) -- node[midway,below] {\tiny $5$} ({0.5*\sq},-2.5);
\draw[ultra thick,gray,->] ({0.5*\sq},-2.5) -- node[midway,below] {\tiny $6$} ({\sq},-2);
\draw[ultra thick,->] ({\sq},-2) -- node[midway,below] {\tiny $7$} ({1.5*\sq},-2.5);
\draw[ultra thick,->] ({1.5*\sq},-2.5) -- node[midway,below] {\tiny $8$} ({2*\sq},-3);
\draw[ultra thick,->] ({2*\sq},-3) -- node[midway,below] {\tiny $9$} ({2.5*\sq},-3.5);
\draw[ultra thick,gray,->] ({2.5*\sq},-3.5) -- node[midway,below] {\tiny $10$} ({3*\sq},-3);
\end{tikzpicture}
\quad
&
\quad
\begin{tikzpicture}[scale=0.7,baseline=(current bounding box.center),>=stealth,rotate=-45]
\draw (2,4) -- (6,4);
\draw (1,3) -- (6,3);
\draw (0,2) -- (3,2);
\draw (0,1) -- (3,1);
\draw (0,0) -- (1,0);
\draw (0,0) -- (0,2);
\draw (1,0) -- (1,3);
\draw (2,1) -- (2,4);
\draw (3,1) -- (3,4);
\draw (4,3) -- (4,4);
\draw (5,3) -- (5,4);
\draw (6,3) -- (6,4);
\node[text centered] at (2.5,3.5) {$3$};
\node[text centered] at (3.5,3.5) {$2$};
\node[text centered] at (4.5,3.5) {$1$};
\node[text centered] at (5.5,3.5) {$1$};
\node[text centered] at (1.5,2.5) {$2$};
\node[text centered] at (2.5,2.5) {$1$};
\node[text centered] at (0.5,1.5) {$2$};
\node[text centered] at (1.5,1.5) {$2$};
\node[text centered] at (2.5,1.5) {$1$};
\node[text centered] at (0.5,0.5) {$1$};
\foreach\x in {1,...,2}{
\draw[dotted] (0,\x) -- (-1,\x+1);
\node at (-1.2,\x+1.2) {\tiny $\lambda^{(\x)}$};
}
\draw[dotted] (1,2) -- (0,3); \node at (-0.2,3.2) {\tiny $\lambda^{(3)}$};
\draw[dotted] (1,3) -- (0,4); \node at (-0.2,4.2) {\tiny $\lambda^{(4)}$};
\draw[dotted] (2,3) -- (1,4); \node at (0.8,4.2) {\tiny $\lambda^{(5)}$};
\foreach\x in {6,...,9}{
\draw[dotted] (\x-4,4) -- (\x-5,5);
\node at (\x-5.2,5.2) {\tiny $\lambda^{(\x)}$};
}
\draw[ultra thick,gray,->] (0,0) -- (0,1);
\draw[ultra thick,gray,->] (0,1) -- (0,2);
\draw[ultra thick,->] (0,2) -- (1,2);
\draw[ultra thick,gray,->] (1,2) -- (1,3);
\draw[ultra thick,->] (1,3) -- (2,3);
\draw[ultra thick,gray,->] (2,3) -- (2,4);
\draw[ultra thick,->] (2,4) -- (3,4);
\draw[ultra thick,->] (3,4) -- (4,4);
\draw[ultra thick,->] (4,4) -- (5,4);
\draw[ultra thick,->] (5,4) -- (6,4);
\draw[ultra thick,->] (0,0) -- (1,0);
\draw[ultra thick,gray,->] (1,0) -- (1,1);
\draw[ultra thick,->] (1,1) -- (2,1);
\draw[ultra thick,->] (2,1) -- (3,1);
\draw[ultra thick,gray,->] (3,1) -- (3,2);
\draw[ultra thick,gray,->] (3,2) -- (3,3);
\draw[ultra thick,->] (3,3) -- (4,3);
\draw[ultra thick,->] (4,3) -- (5,3);
\draw[ultra thick,->] (5,3) -- (6,3);
\draw[ultra thick,gray,->] (6,3) -- (6,4);
\end{tikzpicture}
\end{tabular}
\caption{A skew plane partition and its two-dimensional projection. In this case $M=4$, $N=6$, $S=(+,+,-,+,-,+,-,-,-,-)$ and we have the sequence $\la^{(1)} \subset \la^{(2)} \supset \la^{(3)} \subset \la^{(4)} \supset \la^{(5)} \subset \la^{(6)} \supset \cdots \supset \la^{(9)}$. The support is now given by
$[\la^{(1)} * \cdots * \la^{(9)}] = \nu(T)/\mu(S) = (6,3,3,1) / (2,1)$.}
\label{fig:pp2}
\end{figure}
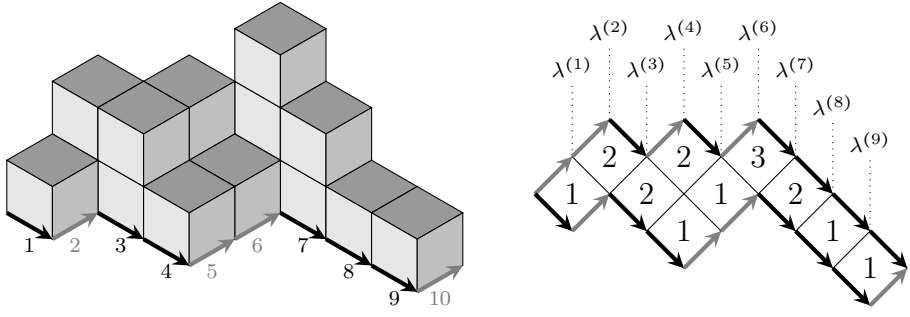

We will be interested in the distribution of this random variable:
\begin{multline}
\label{eq:distrib-hl}
{\rm Prob}^{S}_{M,N}(
[\la^{(1)}* \dots * \la^{(M+N-1)}]
=
\nu/\mu)
:=
\\
\sum_{\la^{(1)},\dots,\la^{(M+N-1)}}
\mathrm{Prob}_{M,N}^S ( \la^{(1)} * \dots * \la^{(M+N-1)} )
\bm{1}_{[\la^{(1)}* \dots * \la^{(M+N-1)}]
=
\nu/\mu}.
\end{multline}

\section{Stochastic six vertex model in a quadrant}
\label{sec:6v}

\subsection{Definition of the model}
\label{sec:6v-defn}

In this section we recall a stochastic inhomogeneous six vertex model in a quadrant considered in \cite{BP1}. We refer to Section 1 of \cite{BP1} for the description of a more general higher spin model.

Let us consider a square lattice $\{ (x,y): x\in \mathbb{Z}_{\ge 1}, y \in \mathbb{Z}_{\ge 1} \}$. A \textit{state} of a six vertex model is a collection of edges from this lattice such that
each of the vertices has one of six types shown in Figure \ref{fig:6types} (in the caption to this figure we also introduce a notation for these types of vertices); for boundary vertices one edge will be given by a boundary condition shown in Figure \ref{fig:stoh6vert_bound}. For convenience, one might orient these edges either up or to the right (depending on whether an edge is vertical or horizontal). Then a state can be described as a collection of up-right paths which cannot go through the same edges.

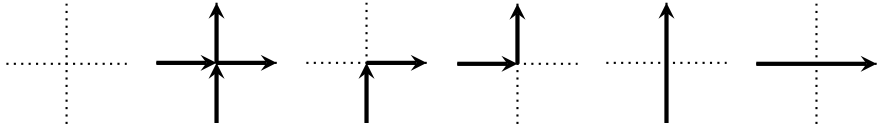
\begin{figure}
\begin{tabular}{cccccc}
\begin{tikzpicture}[>=stealth,scale=0.8]
\draw[dotted,thick] (-1,0) -- (1,0);
\draw[dotted,thick] (0,-1) -- (0,1);
\end{tikzpicture}
&
\begin{tikzpicture}[>=stealth,scale=0.8]
\draw[dotted,thick] (-1,0) -- (1,0);
\draw[dotted,thick] (0,-1) -- (0,1);
\draw[ultra thick,->] (-1,0) -- (0,0);
\draw[ultra thick,->] (0,0) -- (0,1);
\draw[ultra thick,->] (0,-1) -- (0,0);
\draw[ultra thick,->] (0,0) -- (1,0);
\end{tikzpicture}
&
\begin{tikzpicture}[>=stealth,scale=0.8]
\draw[dotted,thick] (-1,0) -- (1,0);
\draw[dotted,thick] (0,-1) -- (0,1);
\draw[ultra thick,->] (0,-1) -- (0,0);
\draw[ultra thick,->] (0,0) -- (1,0);
\end{tikzpicture}
&
\begin{tikzpicture}[>=stealth,scale=0.8]
\draw[dotted,thick] (-1,0) -- (1,0);
\draw[dotted,thick] (0,-1) -- (0,1);
\draw[ultra thick,->] (-1,0) -- (0,0);
\draw[ultra thick,->] (0,0) -- (0,1);
\end{tikzpicture}
&
\begin{tikzpicture}[>=stealth,scale=0.8]
\draw[dotted,thick] (-1,0) -- (1,0);
\draw[dotted,thick] (0,-1) -- (0,1);
\draw[ultra thick,->] (0,-1) -- (0,1);
\end{tikzpicture}
&
\begin{tikzpicture}[>=stealth,scale=0.8]
\draw[dotted,thick] (-1,0) -- (1,0);
\draw[dotted,thick] (0,-1) -- (0,1);
\draw[ultra thick,->] (-1,0) -- (1,0);
\end{tikzpicture}
\end{tabular}
\caption{We denote these 6 types of vertices by symbols $\nul$, $\all$, $\ur$, $\ru$, $\uu$, $\rr$, respectively.}
\label{fig:6types}
\end{figure}

\begin{figure}
\begin{tabular}{cc}
\begin{tikzpicture}[scale=0.7,>=stealth,baseline=(current bounding box.center)]
\foreach\x in {1,...,7}{
\draw[dotted,thick] (\x,0) -- (\x,7);
\node[below] at (\x,0) {\tiny $\x$};
}
\foreach\y in {1,...,6}{
\draw[dotted,thick] (0,\y) -- (8,\y);
\node[left] at (0,\y) {\tiny $\y$};
\draw[ultra thick,->] (0,\y) -- (1,\y);
}
\node[text centered] at (0.2,-0.5) {\tiny $x$};
\node[text centered] at (-0.5,0.2) {\tiny $y$};
\node[text centered] at (-0.2,6.9) {$\vdots$};
\node at (7.9,-0.2) {$\cdots$};
\end{tikzpicture}
\quad
&
\quad
\begin{tikzpicture}[scale=0.7,>=stealth,baseline=(current bounding box.center)]
\foreach\x in {1,...,7}{
\draw[dotted,thick] (\x,0) -- (\x,7);
\node[below] at (\x,0) {\tiny $\x$};
}
\foreach\y in {1,...,6}{
\draw[dotted,thick] (0,\y) -- (8,\y);
\node[left] at (0,\y) {\tiny $\y$};
}
\draw[ultra thick,->] (0,1) -- (5,1); \draw[ultra thick,->] (5,1) -- (5,2); \draw[ultra thick,->] (5,2) -- (6,2); \draw[ultra thick,->] (6,2) -- (6,5); \draw[ultra thick,->] (6,5) -- (7,5); \draw[ultra thick,->] (7,5) -- (7,6); \draw[ultra thick,->] (7,6) -- (8,6);
\draw[ultra thick,->] (0,2) -- (4,2); \draw[ultra thick,->] (4,2) -- (4,3); \draw[ultra thick,->] (4,3) -- (5,3); \draw[ultra thick,->] (5,3) -- (5,5); \draw[ultra thick,->] (5,5) -- (6,5); \draw[ultra thick,->] (6,5) -- (6,6); \draw[ultra thick,->] (6,6) -- (7,6); \draw[ultra thick,->] (7,6) -- (7,7);
\draw[ultra thick,->] (0,3) -- (4,3); \draw[ultra thick,->] (4,3) -- (4,5); \draw[ultra thick,->] (4,5) -- (5,5); \draw[ultra thick,->] (5,5) -- (5,6); \draw[ultra thick,->] (5,6) -- (6,6); \draw[ultra thick,->] (6,6) -- (6,7);
\draw[ultra thick,->] (0,4) -- (3,4); \draw[ultra thick,->] (3,4) -- (3,5); \draw[ultra thick,->] (3,5) -- (4,5); \draw[ultra thick,->] (4,5) -- (4,6); \draw[ultra thick,->] (4,6) -- (5,6); \draw[ultra thick,->] (5,6) -- (5,7);
\draw[ultra thick,->] (0,5) -- (2,5); \draw[ultra thick,->] (2,5) -- (2,6); \draw[ultra thick,->] (2,6) -- (4,6); \draw[ultra thick,->] (4,6) -- (4,7);
\draw[ultra thick,->] (0,6) -- (2,6); \draw[ultra thick,->] (2,6) -- (2,7);
\node[text centered] at (0.2,-0.5) {\tiny $x$};
\node[text centered] at (-0.5,0.2) {\tiny $y$};
\node[text centered] at (-0.2,6.9) {$\vdots$};
\node at (7.9,-0.2) {$\cdots$};
\end{tikzpicture}
\end{tabular}
\caption{Left panel: boundary conditions. Right panel: an example of a state. }
\label{fig:stoh6vert_bound}
\end{figure}
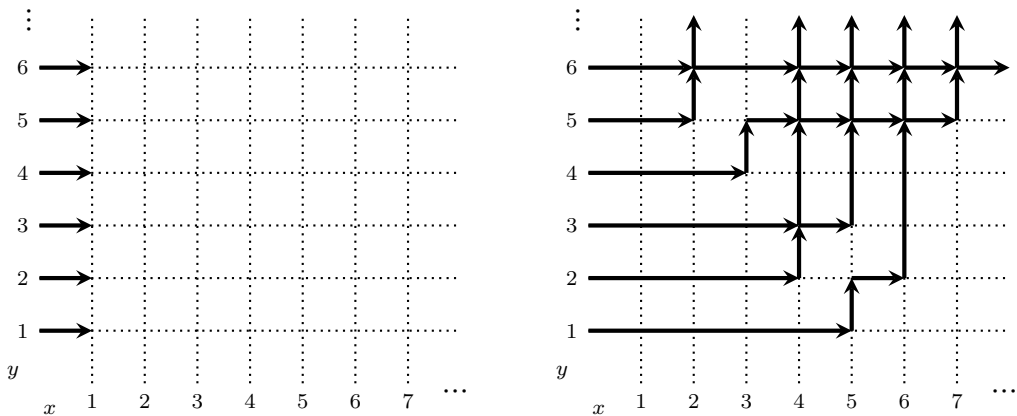

Let us fix a set of parameters $\{ \xi_x \}_{x \in \mathbb{Z}_{\ge 1}}$, $ \{ u_y \}_{y \in \mathbb{Z}_{\ge 1}}$, and $Q$ (which stands for $q$ from \cite{BP1}). We assume that these parameters satisfy
\begin{equation}
\label{eq:cond-6vert}
0<Q<1, \qquad \xi_x >0, u_y >0, \qquad \xi_x u_y > 1 / \sqrt{Q}, \qquad \mbox{for all $x$ and $y$.}
\end{equation}

We define a probability measure on states in the following Markovian way. For any $n\ge 2$ assume that we already have a probability distribution on restrictions of our states to all the vertices $(x,y)$ such that $x+y<n$. Let us increase $n$ by 1. For every point $(x,y)$ such that $x+y=n$ a configuration on vertices with $x+y<n$ specify whether we have entering arrows along the edge at the bottom and/or along the edge to the left of $(x,y)$. Using this information and probability weights
\begin{align*}
& \mathrm{Prob} (\nul) = \mathrm{Prob} (\all) = 1, \qquad \mathrm{Prob} (\rr) = \frac{\frac{1}{Q} - \frac{1}{\sqrt{Q}} \xi_x u_y}{1-\frac{1}{\sqrt{Q}} \xi_x u_y}, \qquad  \mathrm{Prob} (\ru) = \frac{1 - \frac{1}{Q}}{1-\frac{1}{\sqrt{Q}} \xi_x u_y}, \\
& \mathrm{Prob} (\uu) = \frac{1 - \sqrt{Q} \xi_x u_y}{1-\frac{1}{\sqrt{Q}} \xi_x u_y}, \qquad \mathrm{Prob} (\ur) = \frac{(Q-1) \frac{1}{\sqrt{Q}} \xi_x u_y}{1-\frac{1}{\sqrt{Q}} \xi_x u_y},
\end{align*}
we assign a type to the vertex $(x,y)$ (which means we choose the outgoing arrows at the top and to the right of $(x,y)$). Since all these expressions are non-negative, and $\mathrm{Prob} (\rr) + \mathrm{Prob} (\ru)=1$, $ \mathrm{Prob} (\uu) + \mathrm{Prob} (\ur)=1$, this procedure correctly defines a probability measure on all states.

\subsection{$Q$-moments of the height function}

The \textit{height function} $h(x,y)$ is defined as the number of up-right paths which go through or to the right of the point $(x,y)$. It is clear that its value is between $0$ and $y$.
We are interested in the multi-point distribution of the height function at points $(x_1, N), \dots, (x_k,N)$, for any $x_1, \dots, x_k, N$.


\begin{definition}
\label{def:VERTmom}
For positive integers $N,k$, and $m_1 \ge m_2 \ge \dots \ge m_k \ge 1$, a collection of curves $\delta_1$, $\dots$, $\delta_k$ in the complex plane is called \textit{suitable for} $(Q, \{\xi_x \}_{x=1}^{m_1}, \{u_y \}_{y=1}^N)$, if the following conditions hold:
\begin{enumerate}[{\bf 1.}]

\item For any $1 \le r \le k$ the curve $\delta_r$ is a union of finitely many simple positively oriented smooth closed curves.

\item For any $1 \le r \le k$ the curve $\delta_r$ encircles $0, u_1^{-1}, \dots, u_N^{-1}$, and does not encircle any of the points $\{ \sqrt{Q} / \xi_x \}_{x=1}^{m_r}$.

\item For any $1 \le r <s \le k$ the interior of the curve $\delta_r$ contains no points from the curve $Q \delta_s$.

\end{enumerate}

\end{definition}

See Figure \ref{fig:contoursHLand6vert} (right panel) for an example of such contours and parameters.

\begin{proposition}
\label{prop:VERTmom}
Let $N,k$, and $m_1 \ge m_2 \ge \dots \ge m_k \ge 1$ be positive integers. Assume that the collection of curves $\delta_1$, $\dots$, $\delta_k$ is suitable for $(Q, \{\xi_x \}_{x=1}^{m_1}, \{u_y \}_{y=1}^N)$. We have
\begin{multline}
\label{eq:VERTmom}
\mathbf E \prod_{i=1}^k Q^{h (m_i+1,N)} = \frac{ Q^{k(k-1)/2}}{(2 \pi \ii)^k} \oint_{\delta_1} \frac{d w_1}{w_1} \dots \oint_{\delta_k} \frac{d w_k}{w_k} \prod_{1 \le \alpha < \beta \le k} \frac{w_{\alpha} - w_{\beta}}{w_{\alpha} - t w_{\beta}} \\ \times \prod_{l=1}^k \prod_{j=1}^N \frac{1 - Q w_l u_j}{1 - w_l u_j} \prod_{l=1}^k \prod_{i=1}^{m_l} \frac{ \xi_{i} - \frac{1}{\sqrt{Q}} w_{l} }{ \xi_{i} - \sqrt{Q} w_{l}}.
\end{multline}
\end{proposition}

\begin{proof}
This is an analytic continuation of \cite[Theorem 9.8]{BP1}.
\end{proof}

\subsection{Edge distribution on rectangular domains}
\label{sec:edge-rect}

Let us now consider the stochastic six vertex model on a truncated domain: namely, we consider states that live on the finite lattice with vertices at $\{(x,y) : 1 \leq x \leq M, 1 \leq y \leq N\}$:
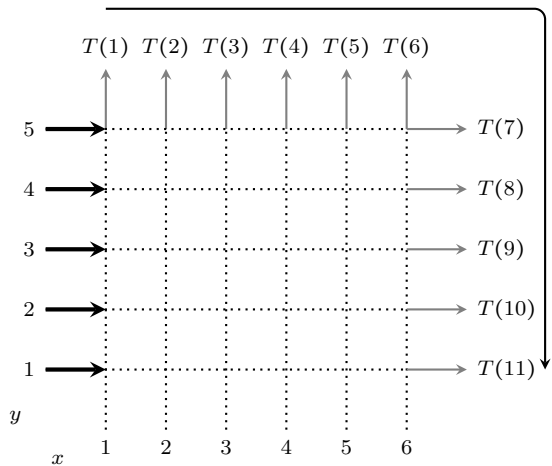
\begin{align*}
\begin{tikzpicture}[scale=0.8,>=stealth,baseline=(current bounding box.center)]
\foreach\x in {1,...,6}{
\node[below] at (\x,0) {\tiny $\x$};
\draw[dotted,thick] (\x,0) -- (\x,5);
\draw[thick,gray,->] (\x,5) -- (\x,6);
}
\foreach\y in {1,...,5}{
\node[left] at (0,\y) {\tiny $\y$};
\draw[ultra thick,->] (0,\y) -- (1,\y);
\draw[dotted,thick] (0,\y) -- (6,\y);
\draw[thick,gray,->] (6,\y) -- (7,\y);
}
\foreach\x in {1,...,6}{
\node[above] at (\x,6) {\tiny $T(\x)$};
}
\foreach\y in {7,...,11}{
\node[right] at (7,12-\y) {\tiny $T(\y)$};
}
\draw[thick,rounded corners,<-] (8.3,1) -- (8.3,7) -- (1,7);
\node[text centered] at (0.2,-0.5) {\tiny $x$}; 
\node[text centered] at (-0.5,0.2) {\tiny $y$}; 
\end{tikzpicture}
\end{align*}
The vertices $(x,N)$, $1 \leq x \leq M-1$ have a single edge that protrudes (upward) out of the region under consideration. Likewise, the vertices $(M,y)$, $1 \leq y \leq N-1$ have a single edge that protrudes (rightward) outside the region, while $(M,N)$ itself has two protruding edges. Collectively, we refer to these as the {\it outgoing edges} of the lattice.

In any state of the model, exactly $N$ of the outgoing edges will be occupied by up-right paths, since the total number of these paths is conserved. Numbering the outgoing edges sequentially from the top left corner of the lattice to the bottom right corner (from $1$ up to $M+N$), let $T(i)=+$ if edge $i$ is unoccupied by a path and $T(i)=-$ if it is occupied, in this way producing a binary string $T = (T(1),\dots,T(M+N))$. Associate to this string the partition $\nu = \nu(T)$. If $\sigma$ denotes a state of the six vertex model on the $M \times N$ lattice, we write $\mathcal{O}(\sigma) = \nu$ to indicate that the outgoing edges of $\sigma$ encode the partition $\nu$.

We shall be interested in the probability distribution
\begin{align*}
{\rm Prob}^{\rm 6V}_{M,N}
(\mathcal{O}(\sigma) = \nu)
:=
\sum_{\sigma}
{\rm Prob}^{\rm 6V}_{M,N} (\sigma)
\bm{1}_{\mathcal{O}(\sigma) = \nu}
\end{align*}
where ${\rm Prob}^{\rm 6V}_{M,N}(\sigma)$ is the probability of the state $\sigma$, which can be obtained either as the product of all local Boltzmann weights appearing in $\sigma$, or through the Markovian procedure described in Section \ref{sec:6v-defn}.

\subsection{Edge distribution on domains with a jagged edge}
\label{sec:jagged}

We can work on a slightly more general class of domains. The playing field is still the positive quadrant $\{ (x,y): x\in \mathbb{Z}_{\ge 1}, y \in \mathbb{Z}_{\ge 1} \}$, but instead of working on the rectangular $M \times N$ domain, we consider the effect of truncating the lattice along some boundary that is not necessarily rectangular (see Figure \ref{fig:jagged}).
\begin{figure}
\begin{tabular}{cc}
\begin{tikzpicture}[scale=0.8,>=stealth,baseline=(current bounding box.center)]
\foreach\x in {1,...,6}{
\draw[dotted,thick] (\x,0) -- (\x,5);
\node[below] at (\x,0) {\tiny $\x$};
}
\foreach\y in {1,...,5}{
\draw[dotted,thick] (0,\y) -- (6,\y);
\node[left] at (0,\y) {\tiny $\y$};
}
\draw[fill opacity=0.5,fill=white!90!black] (0,5) -- (2,5) -- (2,4) -- (4,4) -- (4,2) -- (5,2) -- (5,1) -- (6,1) -- (6,0) -- (6,5) -- cycle;
\draw[ultra thick,<-] (6,0) -- (6,1);
\draw[ultra thick,gray,<-] (6,1) -- (5,1);
\draw[ultra thick,<-] (5,1) -- (5,2);
\draw[ultra thick,gray,<-] (5,2) -- (4,2);
\draw[ultra thick,<-] (4,2) -- (4,3);
\draw[ultra thick,<-] (4,3) -- (4,4);
\draw[ultra thick,gray,<-] (4,4) -- (3,4);
\draw[ultra thick,gray,<-] (3,4) -- (2,4);
\draw[ultra thick,<-] (2,4) -- (2,5);
\draw[ultra thick,gray,<-] (2,5) -- (1,5);
\draw[ultra thick,gray,<-] (1,5) -- (0,5);
\node[text centered] at (0.2,-0.5) {\tiny $x$};
\node[text centered] at (-0.5,0.2) {\tiny $y$};
\end{tikzpicture}
\quad
&
\quad
\begin{tikzpicture}[scale=0.8,>=stealth,baseline=(current bounding box.center)]
\foreach\x in {1,...,6}{
\node[below] at (\x,0) {\tiny $\x$};
}
\draw[dotted,thick] (6,0) -- (6,1);
\draw[dotted,thick] (5,0) -- (5,2);
\draw[dotted,thick] (4,0) -- (4,4);
\draw[dotted,thick] (3,0) -- (3,4);
\draw[dotted,thick] (2,0) -- (2,5);
\draw[dotted,thick] (1,0) -- (1,5);
\foreach\y in {1,...,5}{
\node[left] at (0,\y) {\tiny $\y$};
\draw[ultra thick,->] (0,\y) -- (1,\y);
}
\draw[dotted,thick] (0,5) -- (2,5);
\draw[dotted,thick] (0,4) -- (4,4);
\draw[dotted,thick] (0,3) -- (4,3);
\draw[dotted,thick] (0,2) -- (5,2);
\draw[dotted,thick] (0,1) -- (6,1);
\node[text centered] at (0.2,-0.5) {\tiny $x$};
\node[text centered] at (-0.5,0.2) {\tiny $y$};
\draw[thick,gray,->] (6,1) -- (6.9,1);
\draw[thick,gray,->] (6,1) -- (6,1.9);
\draw[thick,gray,->] (5,2) -- (5.9,2);
\draw[thick,gray,->] (5,2) -- (5,2.9);
\draw[thick,gray,->] (4,3) -- (4.9,3);
\draw[thick,gray,->] (4,4) -- (4.9,4);
\draw[thick,gray,->] (4,4) -- (4,4.9);
\draw[thick,gray,->] (3,4) -- (3,4.9);
\draw[thick,gray,->] (2,5) -- (2.9,5);
\draw[thick,gray,->] (2,5) -- (2,5.9);
\draw[thick,gray,->] (1,5) -- (1,5.9);
\draw[thick,rounded corners,<-] (7.5,1) -- (7.5,6.5) -- (1,6.5);
\node at (6.5,5.5) {$T$};
\end{tikzpicture}
\end{tabular}
\caption{Left panel: the Young diagram $\mu = (4,3,1,1)$ used to cut a jagged boundary. Right panel: the resulting jagged sub-domain of the $M \times N$ lattice. Outgoing edges shown in grey.}
\label{fig:jagged}
\end{figure}
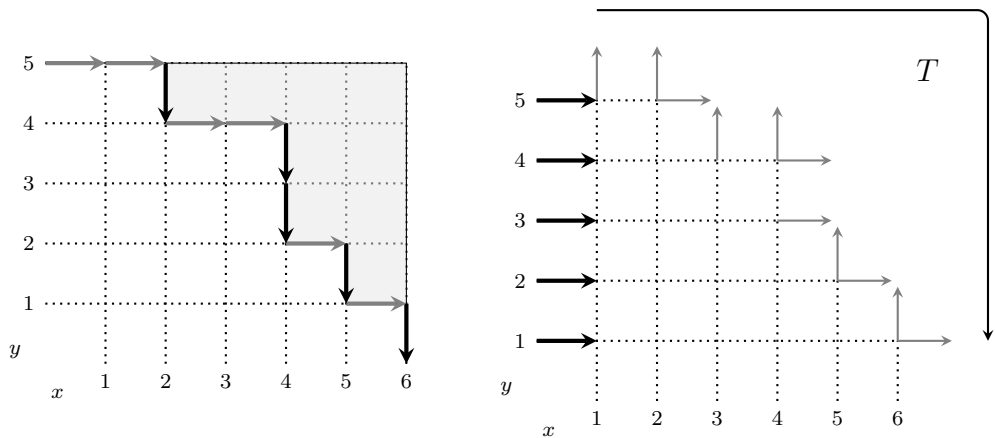
We call such domains {\it jagged}. A jagged domain is specified by {\bf 1.} The two positive integers $M$ and $N$; {\bf 2.} A binary string $S = (S(1),\dots,S(M+N))$ consisting of $M$ pluses and $N$ minuses, that determines the jagged edge. Starting at the point $(0,N)$, one begins tracing a path consisting of $M$ right steps and $N$ down steps, where the $i$-th step is right if $S(i) = +$ and down if $S(i) = -$. The final coordinate of this path is necessarily $(M,0)$. By convention, we always take the first step to be right and the final step to be down, so that $S(1)=+$ and $S(M+N)=-$. Deleting the part of the lattice lying above this path, we obtain a jagged domain. As illustrated in the left panel of Figure \ref{fig:jagged}, the excised region is nothing but the Young diagram of the partition $\mu(S)$, rotated clockwise by 90 degrees. The domains that we considered so far, which were rectangular, corresponded with $S=(+,\dots,+,-,\dots,-)$, for which $\mu(S) = \varnothing$.

A jagged sub-domain of the $M \times N$ lattice continues to possess $M+N$ outgoing edges; \ie external vertex edges that protrude upward or rightward. As before, exactly $N$ of these will be occupied by up-right paths, and numbering the edges in clockwise fashion (as shown in the right panel of Figure \ref{fig:jagged}) we can again obtain a binary string $T$ and corresponding partition $\nu(T)$ by following the same conventions as in Section \ref{sec:edge-rect}. Extending the notation of that section, we write $\mathcal{O}(\sigma) = \nu/\mu$ to indicate that $\sigma$ is a state on the jagged domain encoded by $\mu$, with outgoing edges that encode the partition $\nu$.

We are then interested in the probability distribution of this random variable:
\begin{align}
\label{eq:distrib-6v}
{\rm Prob}^{\rm 6V}_{M,N}
(\mathcal{O}(\sigma) = \nu/\mu)
:=
\sum_{\sigma}
{\rm Prob}^{\rm 6V}_{M,N} (\sigma)
\bm{1}_{\mathcal{O}(\sigma) = \nu/\mu}.
\end{align}

\section{Matching distributions}

\subsection{Matching $q$-moments of partition lengths and height functions}
\label{sc:moments-match}

The Hall--Littlewood process considered in Section \ref{sec:HL-process-usual} depends on $t$ and parameters $\{a_i \}_{i \ge 1}$, $\{ b_j \}_{j \ge 1}$. The stochastic six vertex model from Section \ref{sec:6v-defn} depends on $Q$, $\{\xi_x \}_{x \ge 1}$, $ \{ u_y \}_{y \ge 1}$. We want to show that these stochastic processes are closely related when the parameters satisfy
\begin{equation}
\label{eq:param-match}
t=Q, \qquad a_x= \frac{\sqrt{t}}{\xi_x}, \ \ \mbox{for all $x$}, \qquad b_y = \frac{1}{t u_y}, \ \ \mbox{for all $y$.}
\end{equation}

\begin{theorem}
\label{th:match-param}
Assume that the parameters $(t, \{a_i\}, \{ b_j \} )$ and $( Q, \{\xi_i \}, \{ u_j \} )$ satisfy the conditions \eqref{eq:param-match}. For any positive integers $N,k$, and $m_1, \dots, m_k$ the vector $(N-\lambda'_1 (m_1,N), \dots, N - \lambda'_1 (m_k,N))$ has the same distribution as $(h(m_1+1,N), \dots, h(m_k+1,N))$.
\end{theorem}

In this section we establish this theorem with the use of moments.

\begin{lemma}
\label{lem:match-param}
Assume that the parameters $(t, \{a_i\}, \{ b_j \} )$ and $( Q, \{\xi_i \}, \{ u_j \} )$ satisfy the conditions \eqref{eq:param-match}. Then, for $m_1 \ge m_2 \ge \dots \ge m_k \ge 1$, we have
\begin{equation}
\label{eq:main-match-eq}
\mathbf E \left( t^{(N - \lambda'_1 (m_1,N)) + (N - \lambda'_1 (m_2,N)) + \dots + (N - \lambda'_1 (m_k,N)) } \right) = \mathbf E \left( Q^{h(m_1+1,N) + h(m_2+1,N) +\dots + h(m_k+1,N)} \right).
\end{equation}
\end{lemma}

\begin{proof}

First, note that the conditions \eqref{eq:cond-HL} and \eqref{eq:cond-6vert} are just the same with the matching \eqref{eq:param-match}. We start by proving the lemma with further restrictions on parameters.
Namely, let us assume that the parameters $a_1, \dots, a_{m_1}$ are less than a certain real number $A$, the parameters $b_1, \dots, b_N$ are such that the points $\{1 / (t b_j) \}_{j=1}^N$ are greater than a certain real number $B$, and that $A < \tilde t^k B$, for some $\tilde t < t$. It is easy to see that such $A$ and $B$ exist if we have $a_x b_y < t^{k-1}$ for all $x$ and $y$. Then the contours $\gamma_r := \{ z: |z|=\tilde t^{r-1} B \}$, $1 \le r \le k$, are suitable for $(t, \{a_i\}, \{ b_j \} )$ in the sense of Definition \ref{def:contours-HL}.

From the relations between the parameters, we have
$$
\frac{1 - b_y / w }{1 - t b_y/w} = t^{-1} \frac{1 - t u_y w}{1 - u_y w}, \qquad
\frac{t / w -a_x}{1/w - a_x} = t \frac{\xi_x - w / \sqrt{t}}{\xi_x - \sqrt{t} w},
$$
$$
\prod_{1 \le i<j \le k} \frac{1 / w_i - 1 / w_j}{t / w_i - 1 / w_j} = \prod_{1 \le i<j \le k} \frac{w_i - w_j}{w_i - t w_j}.
$$
Let us make a change of variables $w_i := 1 / z_i$, $i=1, \dots, k$, in the right hand side of \eqref{eq:HLmom}. Note that after such a change the image of the contour $\gamma_r$ will encircle $0$ and $t b_y = u_y^{-1}$, $y=1, \dots, N$, but not $ 1 / a_x = \xi_x / \sqrt{t}$, $x=1, \dots, m_r$, which exactly coincides with the condition on contours from Definition \ref{def:VERTmom}. Also, after this the contours are embedded into each other exactly as in Definition \ref{def:VERTmom} (see Figure \ref{fig:contoursHLand6vert}).

Combining all these facts and applying Proposition \ref{prop:VERTmom}, we have
$$
\mathbf E \left( t^{(m_1 - \lambda'_1 (m_1,N)) + (m_2 - \lambda'_1 (m_2,N)) + \dots + (m_k - \lambda'_1 (m_k,N)) } \right) = t^{m_1+\dots +m_k - kN} \mathbf E \prod_{i=1}^k t^{h (m_i+1,N)}.
$$
This implies the proposition \eqref{eq:main-match-eq} for the set of parameters with $a_x b_y < t^{k-1}$, $x=1, \dots, m_1$, $y=1, \dots,N$.

Let us prove the lemma for all parameters $\{ a_x \}, \{b_y\}$ satisfying \eqref{eq:cond-HL}. Note that in variables $a_1, \dots, a_{m_1}$, $b_1, \dots, b_N$ the left hand side of \eqref{eq:main-match-eq} is the sum of an absolutely convergent series in a neighborhood of $0^{m_1+N}$, while the right hand side of \eqref{eq:main-match-eq} is a rational function. Since these functions coincide on an open set in $\mathbb R^{m_1+n}$, they must coincide everywhere.

\end{proof}

\begin{proof}[Proof of Theorem \ref{th:match-param}]

Lemma \ref{lem:match-param} implies that the random vectors $(t^{N-\lambda'_1 (m_1,N)}, \dots, t^{N - \lambda'_1 (m_k,N)})$ and $(t^{h(m_1+1,N)}, \dots, t^{h(m_k+1,N)})$ have the same joint moments. Since these random vectors take values in $[0,1]^k$, we conclude that $(t^{N-\lambda'_1 (m_1,N)}, \dots, t^{N - \lambda'_1 (m_k,N)})$ and $(t^{h(m_1+1,N)}, \dots, t^{h(m_k+1,N)})$ have the same distribution. Taking logarithms, we arrive at the statement of the theorem.
\end{proof}

\subsection{A more general result on matching distributions}
\label{sec:more-gen-th}

For positive integers $M,N$ let $S \in \mathcal{S}_{M,N}^{+}$ be a string of signs as defined in Section \ref{sec:general-HL}. Consider a collection of points $((x_0 (S), y_0(S)), \dots, (x_{M+N} (S),y_{M+N} (S)))$ in $\mathbb Z_{\ge 0}^2$ such that:
\begin{enumerate}[{\bf 1.}]
\item $(x_0 (S) ,y_0 (S)) = (0,N)$, $(x_{M+N} (S), y_{M+N} (S) ) = (M,0)$.
\item If $S(i)=+$, then $x_{i} (S) = x_{i-1} (S) +1$, $y_{i} (S) = y_{i-1} (S)$. If $S(i)=-$, then $x_{i} (S)=x_{i-1} (S)$, $y_{i} (S)=y_{i-1} (S)-1$.
\end{enumerate}
It is clear that $S$ uniquely determines such a collection. These points are just the positions of the vertices visited, when cutting a jagged edge formed by the partition $\mu(S)$; see Section \ref{sec:jagged}.

\begin{theorem}
\label{th:equiv-distrib}
For positive integers $M,N$ and a string $S \in \mathcal{S}_{M,N}^{+}$, the random vectors $\{ \lambda'_1 (i,N,S) \}_{i=1}^{M+N-1}$ defined in Section \ref{sec:general-HL} and $\{ y_i(S) - h (x_i (S)+1,y_i(S)) \}_{i=1}^{M+N-1}$ are identically distributed.
\end{theorem}

If the string $S$ equals $(+,\dots, +, -, \dots, -)$, this statement for the first $M$ coordinates of the random vectors is equivalent to Theorem \ref{th:match-param}. We will defer the proof of this statement to the next section, where we prove an equivalent statement, that the distributions \eqref{eq:distrib-hl} and \eqref{eq:distrib-6v} match.

\section{$t$-Bosons and the proof of Theorem \ref{th:equiv-distrib}}

\subsection{Weights of the six vertex model after parameter-matching \eqref{eq:param-match}}

Let us begin by recalling the weights assigned to the six vertices in Figure \ref{fig:6types}, where for convenience we impose the matching \eqref{eq:param-match}. The probability weights become
\begin{align}
\mathrm{Prob} (\nul) = \mathrm{Prob} (\all) = 1,
\qquad
\begin{split}
& \mathrm{Prob} (\rr) = \frac{1-a_x b_y}{1- t a_x b_y},
\qquad
\mathrm{Prob} (\ru) = \frac{(1-t) a_x b_y}{1-t a_x b_y},
\\
& \mathrm{Prob} (\uu) = \frac{t(1-a_x b_y)}{1-t a_x b_y},
\qquad
\mathrm{Prob} (\ur) =
\frac{1-t}{1-t a_x b_y}.
\end{split}
\label{eq:six-vertex-weights}
\end{align}
At all times in this section, we will assume that the Boltzmann weights of the six vertex model are given by \eqref{eq:six-vertex-weights}.

\subsection{An integrable model of $t$-bosons}

We now introduce a further integrable vertex model, living on the lattice $\mathbb{Z}^2$. Horizontal edges of the lattice can be occupied by at most one lattice path, but no restriction is imposed on the number of paths that traverse a vertical edge. Every intersection of horizontal and vertical gridlines constitutes a vertex, and each vertex is assigned a Boltzmann weight that depends on the local configuration of lattice paths about that intersection. Assuming conservation of lattice paths through a vertex, four types of vertices are possible. We indicate these vertices and their explicit weights below:
\begin{align}
\label{eq:black-vertices}
\begin{array}{cccc}
\begin{tikzpicture}[scale=0.8,>=stealth]
\draw[lgray,ultra thick] (-1,0) -- (1,0);
\draw[lgray,line width=10pt] (0,-1) -- (0,1);
\node[below] at (0,-1) {$m$};
\draw[ultra thick,->,rounded corners] (-0.075,-1) -- (-0.075,1);
\draw[ultra thick,->,rounded corners] (0.075,-1) -- (0.075,1);
\node[above] at (0,1) {$m$};
\end{tikzpicture}
\quad\quad\quad
&
\begin{tikzpicture}[scale=0.8,>=stealth]
\draw[lgray,ultra thick] (-1,0) -- (1,0);
\draw[lgray,line width=10pt] (0,-1) -- (0,1);
\node[below] at (0,-1) {$m$};
\draw[ultra thick,->,rounded corners] (-0.075,-1) -- (-0.075,1);
\draw[ultra thick,->,rounded corners] (0.075,-1) -- (0.075,0) -- (1,0);
\node[above] at (0,1) {$m-1$};
\end{tikzpicture}
\quad\quad\quad
&
\begin{tikzpicture}[scale=0.8,>=stealth]
\draw[lgray,ultra thick] (-1,0) -- (1,0);
\draw[lgray,line width=10pt] (0,-1) -- (0,1);
\node[below] at (0,-1) {$m$};
\draw[ultra thick,->,rounded corners] (-1,0) -- (-0.15,0) -- (-0.15,1);
\draw[ultra thick,->,rounded corners] (0,-1) -- (0,1);
\draw[ultra thick,->,rounded corners] (0.15,-1) -- (0.15,1);
\node[above] at (0,1) {$m+1$};
\end{tikzpicture}
\quad\quad\quad
&
\begin{tikzpicture}[scale=0.8,>=stealth]
\draw[lgray,ultra thick] (-1,0) -- (1,0);
\draw[lgray,line width=10pt] (0,-1) -- (0,1);
\node[below] at (0,-1) {$m$};
\draw[ultra thick,->,rounded corners] (-1,0) -- (-0.15,0) -- (-0.15,1);
\draw[ultra thick,->,rounded corners] (0,-1) -- (0,1);
\draw[ultra thick,->,rounded corners] (0.15,-1) -- (0.15,0) -- (1,0);
\node[above] at (0,1) {$m$};
\end{tikzpicture}
\\
1
\quad\quad\quad
&
a
\quad\quad\quad
&
(1-t^{m+1})
\quad\quad\quad
&
a
\end{array}
\end{align}
where $a$ is a local (spectral) parameter associated to the vertex, and $t$ is a global parameter of the model. This is the lattice realization of the $t$-boson model (see, for example, \cite{BB,BIK}), and its connection with Hall--Littlewood polynomials is well known \cite{T,K,WZ}. It can also be recovered from the higher-spin vertex model of \cite{B,BP1}, in the limit of infinite spin (in the notation of \cite{B,BP1}, this is the limit $s \rightarrow 0$).

\begin{theorem}
\label{thm:RLL}
For any fixed $0 \leq i_1,i_2,j_1,j_2 \leq 1$ and $m,n \in \mathbb{Z}_{\geq 0}$, the Yang--Baxter equation holds:
\begin{align}
\label{eq:RLL}
\sum_{0 \leq k_1,k_2 \leq 1}\
\sum_{p=0}^{\infty}\ \ \
\begin{tikzpicture}[baseline=(current bounding box.center),scale=0.8]
\draw[dotted,thick] (-2,1) node[left] {$i_1$} -- (-1,0) node[below] {$k_1$};
\draw[dotted,thick] (-2,0) node[left] {$i_2$} -- (-1,1) node[above] {$k_2$};
\draw[lgray,ultra thick] (-1,1) -- (1,1) node[right,black] {$j_2$};
\draw[lgray,ultra thick] (-1,0) -- (1,0) node[right,black] {$j_1$};
\draw[lgray,line width=10pt] (0,-1) -- (0,2);
\node[below] at (0,-1) {$m$};
\node at (0,0.5) {$p$};
\node[above] at (0,2) {$n$};
\draw[thin,dashed,->] (0,0) -- (1,-1) node[right] {$b^{-1}$};
\draw[thin,dashed,->] (0,1) -- (1,2) node[right] {$a$};
\end{tikzpicture}
\quad
=
\quad
\sum_{0 \leq k_1,k_2 \leq 1}\
\sum_{p=0}^{\infty}\ \ \
\begin{tikzpicture}[baseline=(current bounding box.center),scale=0.8]
\draw[dotted,thick] (1,1) node[above] {$k_1$} -- (2,0) node[right] {$j_1$};
\draw[dotted,thick] (1,0) node[below] {$k_2$} -- (2,1) node[right] {$j_2$};
\draw[lgray,ultra thick] (-1,1) node[left,black] {$i_1$} -- (1,1);
\draw[lgray,ultra thick] (-1,0) node[left,black] {$i_2$} -- (1,0);
\draw[lgray,line width=10pt] (0,-1) -- (0,2);
\node[below] at (0,-1) {$m$};
\node at (0,0.5) {$p$};
\node[above] at (0,2) {$n$};
\draw[thin,dashed,->] (0,0) -- (-1,-1) node[left] {$a$};
\draw[thin,dashed,->] (0,1) -- (-1,2) node[left] {$b^{-1}$};
\end{tikzpicture}
\end{align}
where the spectral parameters of the bosonic vertices are indicated on the picture, and the diagonally attached vertices are vertices in the stochastic six vertex model of \eqref{eq:six-vertex-weights}, rotated clockwise by 45 degrees.
\end{theorem}

\begin{proof}
This is by direct computation, since there are only sixteen relations to verify (all possible choices of $i_1,i_2,j_1,j_2$), treating $m$ and $n$ as arbitrary non-negative integers.
\end{proof}

It is important to introduce an alternative normalization of the vertex weights \eqref{eq:black-vertices}, obtained by sending $a \rightarrow b^{-1}$ and then simply multiplying all vertices by $b$:
\begin{align}
\label{eq:red-vertices}
\begin{array}{cccc}
\begin{tikzpicture}[scale=0.8,>=stealth]
\draw[lred,ultra thick] (-1,0) -- (1,0);
\draw[lred,line width=10pt] (0,-1) -- (0,1);
\node[below] at (0,-1) {$m$};
\draw[ultra thick,->,rounded corners] (-0.075,-1) -- (-0.075,1);
\draw[ultra thick,->,rounded corners] (0.075,-1) -- (0.075,1);
\node[above] at (0,1) {$m$};
\end{tikzpicture}
\quad\quad\quad
&
\begin{tikzpicture}[scale=0.8,>=stealth]
\draw[lred,ultra thick] (-1,0) -- (1,0);
\draw[lred,line width=10pt] (0,-1) -- (0,1);
\node[below] at (0,-1) {$m$};
\draw[ultra thick,->,rounded corners] (-0.075,-1) -- (-0.075,1);
\draw[ultra thick,->,rounded corners] (0.075,-1) -- (0.075,0) -- (1,0);
\node[above] at (0,1) {$m-1$};
\end{tikzpicture}
\quad\quad\quad
&
\begin{tikzpicture}[scale=0.8,>=stealth]
\draw[lred,ultra thick] (-1,0) -- (1,0);
\draw[lred,line width=10pt] (0,-1) -- (0,1);
\node[below] at (0,-1) {$m$};
\draw[ultra thick,->,rounded corners] (-1,0) -- (-0.15,0) -- (-0.15,1);
\draw[ultra thick,->,rounded corners] (0,-1) -- (0,1);
\draw[ultra thick,->,rounded corners] (0.15,-1) -- (0.15,1);
\node[above] at (0,1) {$m+1$};
\end{tikzpicture}
\quad\quad\quad
&
\begin{tikzpicture}[scale=0.8,>=stealth]
\draw[lred,ultra thick] (-1,0) -- (1,0);
\draw[lred,line width=10pt] (0,-1) -- (0,1);
\node[below] at (0,-1) {$m$};
\draw[ultra thick,->,rounded corners] (-1,0) -- (-0.15,0) -- (-0.15,1);
\draw[ultra thick,->,rounded corners] (0,-1) -- (0,1);
\draw[ultra thick,->,rounded corners] (0.15,-1) -- (0.15,0) -- (1,0);
\node[above] at (0,1) {$m$};
\end{tikzpicture}
\\
b
\quad\quad\quad
&
1
\quad\quad\quad
&
b (1-t^{m+1})
\quad\quad\quad
&
1
\end{array}
\end{align}
We use a red background to indicate that this normalization is employed, rather than that of \eqref{eq:black-vertices}.

\subsection{Row-to-row operators and their exchange relations}
\label{sec:row-ops}

In what follows, we let $V_i$ be an infinite dimensional vector space with basis vectors $\{\ket{m}_i\}_{m \in \mathbb{Z}_{\geq0}}$. It has a dual space $V_i^{*}$ spanned by $\{\bra{m}_i\}_{m \in \mathbb{Z}_{\geq0}}$, where $\bra{m}_i \ket{n}_i = \delta_{m,n}$ for all $m,n \in \mathbb{Z}_{\geq 0}$. Further, we let $V_{1\dots L}$ denote the $L$-fold tensor product $\bigotimes_{i =1}^{L} V_i$.

Joining $L$ of the vertices \eqref{eq:black-vertices} horizontally (with common spectral parameter $a$) and summing over all possible states on internal horizontal edges, we obtain a {\it row vertex.} We denote its Boltzmann weight\footnote{Note that, once all external states are specified, each internal horizontal edge is forced to assume a unique value. The weight \eqref{eq:row-vert} is therefore fully factorized.} by
\begin{align}
\label{eq:row-vert}
w_a
\left(
\begin{tikzpicture}[scale=0.5,baseline=(current bounding box.center)]
\draw[lgray,thick] (-1,0) -- (6,0);
\node[left] at (-0.8,0) {\tiny $i$};\node[right] at (5.8,0) {\tiny $j$};
\foreach\x in {0,...,5}{
\draw[lgray,line width=7pt] (\x,-1) -- (\x,1);
}
\node[below,text centered] at (0,-0.8) {\tiny $m_{L}$};\node[above] at (0,0.8) {\tiny $n_{L}$};
\node[below,text centered] at (3,-0.8) {$\cdots$};\node[above] at (3,0.8) {$\cdots$};
\node[below,text centered] at (5,-0.8) {\tiny $m_1$};\node[above] at (5,0.8) {\tiny $n_1$};
\end{tikzpicture}
\right)
\equiv
w_a\Big(i,\{m_1,\dots,m_{L}\} \Big| j,\{n_1,\dots,n_{L}\}\Big).
\end{align}
We then construct row-to-row operators that act linearly on $V_{1\dots L}$ as follows:
\begin{align*}
T_a(i|j) : \ket{n_1}_1 \otimes \cdots \otimes \ket{n_{L}}_{L}
\mapsto
\sum_{m_1,\dots,m_{L} \geq 0}
w_a\Big(i,\{m_1,\dots,m_{L}\} \Big| j,\{n_1,\dots,n_{L}\}\Big)
\ket{m_1}_1 \otimes \cdots \otimes \ket{m_{L}}_{L}.
\end{align*}
There are in total four such operators, corresponding to all possible values of $0 \leq i,j \leq 1$. It is more conventional to label them alphabetically, by writing
\begin{align}
\label{eq:row-ops}
\begin{pmatrix}
T_a(0|0) & T_a(0|1)
\\ \\
T_a(1|0) & T_a(1|1)
\end{pmatrix}
\equiv
\begin{pmatrix}
A_{L}(a) & B_{L}(a)
\\ \\
C_{L}(a) & D_{L}(a)
\end{pmatrix};
\end{align}
this is known as the {\it monodromy matrix}\/ in the context of the algebraic Bethe ansatz.
\begin{corollary}
\label{thm:RTT}
For any fixed $0 \leq i_1,i_2,j_1,j_2 \leq 1$ and $\{m_1,\dots,m_{L}\},\{n_1,\dots,n_{L}\} \in \mathbb{Z}_{\geq 0}$, we have the following relation:
\begin{multline}
\label{eq:RTT}
\sum_{0 \leq k_1,k_2 \leq 1}\
\sum_{p_1,\dots,p_{L}\geq 0}\ \ \
\begin{tikzpicture}[baseline=(current bounding box.center),scale=0.8]
\draw[dotted,thick] (-2,1) node[left] {$i_1$} -- (-1,0) node[below] {$k_1$};
\draw[dotted,thick] (-2,0) node[left] {$i_2$} -- (-1,1) node[above] {$k_2$};
\draw[lgray,ultra thick] (-1,1) -- (4,1) node[right,black] {$j_2$};
\draw[lgray,ultra thick] (-1,0) -- (4,0) node[right,black] {$j_1$};
\foreach\x in {0,...,3}{
\draw[lgray,line width=10pt] (3-\x,-1) -- (3-\x,2);
}
\node[below] at (3,-1) {$m_1$};
\node at (3,0.5) {$p_1$};
\node[above] at (3,2) {$n_1$};
\node[below] at (2,-1) {$m_2$};
\node at (2,0.5) {$p_2$};
\node[above] at (2,2) {$n_2$};
\node[below] at (0,-1) {$m_{L}$};
\node at (0,0.5) {$p_{L}$};
\node[above] at (0,2) {$n_{L}$};
\node[text centered] at (1,0.5) {$\cdots$};
\draw[thin,dashed,->] (3,0) -- (4,-1) node[right] {$b^{-1}$};
\draw[thin,dashed,->] (3,1) -- (4,2) node[right] {$a$};
\end{tikzpicture}
\quad
=
\\
\quad
\sum_{0 \leq k_1,k_2 \leq 1}\
\sum_{p_1,\dots,p_{L}\geq 0}\ \ \
\begin{tikzpicture}[baseline=(current bounding box.center),scale=0.8]
\foreach\x in {0,...,3}{
\draw[lgray,line width=10pt] (3-\x,-1) -- (3-\x,2);
}
\draw[lgray,ultra thick] (-1,1) node[left,black] {$i_1$} -- (4,1);
\draw[lgray,ultra thick] (-1,0) node[left,black] {$i_2$} -- (4,0);
\draw[dotted,thick] (4,1) node[above] {$k_1$} -- (5,0) node[right] {$j_1$};
\draw[dotted,thick] (4,0) node[below] {$k_2$} -- (5,1) node[right] {$j_2$};
\node[below] at (3,-1) {$m_1$};
\node at (3,0.5) {$p_1$};
\node[above] at (3,2) {$n_1$};
\node[below] at (2,-1) {$m_2$};
\node at (2,0.5) {$p_2$};
\node[above] at (2,2) {$n_2$};
\node[below] at (0,-1) {$m_{L}$};
\node at (0,0.5) {$p_{L}$};
\node[above] at (0,2) {$n_{L}$};
\node[text centered] at (1,0.5) {$\cdots$};
\draw[thin,dashed,->] (0,0) -- (-1,-1) node[left] {$a$};
\draw[thin,dashed,->] (0,1) -- (-1,2) node[left] {$b^{-1}$};
\end{tikzpicture}
\end{multline}
where all vertices in a row share a common spectral parameter, as indicated.
\end{corollary}

\begin{proof}
This follows by the {\it unzipping} argument in integrable models, namely, by repeated application of the local relation \eqref{eq:RLL}.
\end{proof}
Like equation \eqref{eq:RLL}, the global Yang--Baxter equation \eqref{eq:RTT} gives rise to sixteen possible exchange relations among the row-to-row operators \eqref{eq:row-ops}. Below we list the ones corresponding to $(i_1,i_2) = (1,0)$, and four possible choices for $(j_1,j_2)$, since we will need them in what follows:
\begin{equation}
\label{eq:exchange1}
\begin{aligned}
(1-ab)
C_{L}(b^{-1}) A_{L}(a)
+
ab (1-t)
A_{L}(b^{-1}) C_{L}(a)
&=
(1-tab)
A_{L}(a) C_{L}(b^{-1}),
\\
(1-ab)
C_{L}(b^{-1}) B_{L}(a)
+
ab (1-t)
A_{L}(b^{-1}) D_{L}(a)
&=
\\
t(1-ab)
B_{L}(a) & C_{L}(b^{-1})
+
ab (1-t)
A_{L}(a) D_{L}(b^{-1}),
\\
(1-ab)
D_{L}(b^{-1}) A_{L}(a)
+
ab(1-t)
B_{L}(b^{-1}) C_{L}(a)
&=
\\
(1-ab)
A_{L}(a) & D_{L}(b^{-1})
+
(1-t)
B_{L}(a) C_{L}(b^{-1}),
\\
(1-ab)
D_{L}(b^{-1}) B_{L}(a)
+
ab (1-t)
B_{L}(b^{-1}) D_{L}(a)
&=
(1-tab)
B_{L}(a) D_{L}(b^{-1}).
\end{aligned}
\end{equation}

\subsection{Exchange relations in the limit of infinite volume}

Next we will send $L$, the horizontal dimension of the lattice, to infinity.
If $a$ is a generic parameter, some care is needed in order for the row-to-row operators to be well-defined in this limit. More specifically, both $\lim_{L \rightarrow \infty} A_L(a)$ and $\lim_{L \rightarrow \infty} B_L(a)$ make sense as formal power series in $a$, but $\lim_{L \rightarrow \infty} C_L(a)$ and $\lim_{L \rightarrow \infty} D_L(a)$ do not.

A partial cure to this problem is to replace every vertex that appears by its red counterpart \eqref{eq:red-vertices}. We denote the resulting row-to-row operators using a bar, \ie by $\b{A}_{L}(b)$, $\b{B}_{L}(b)$, $\b{C}_{L}(b)$ and $\b{D}_{L}(b)$. In that case, $\lim_{L \rightarrow \infty} \b{C}_{L}(b)$ and $\lim_{L \rightarrow \infty} \b{D}_{L}(b)$ now make sense as formal power series in $b$, while $\lim_{L \rightarrow \infty} \b{A}_{L}(b)$ and $\lim_{L \rightarrow \infty} \b{B}_{L}(b)$ do not.

Clearly in the infinite volume limit, we can only ever make sense of half of the row vertices as formal power series. On the other hand, if the spectral parameter is taken to be real, row-to-row operators with infinite degree in this parameter might then converge to zero. Returning to the equations \eqref{eq:exchange1}, we multiply each by $b^{L}$, before taking $L \rightarrow \infty$. Assuming that $a$ and $b$ are real parameters satisfying $|ab| < 1$, we find that the products of operators $\b{A}_{L}(b) C_{L}(a)$, $\b{A}_{L}(b) D_{L}(a)$, $\b{B}_{L}(b) C_{L}(a)$ and $\b{B}_{L}(b) D_{L}(a)$ become vanishingly small in the limit, and hence
\begin{equation}
\label{eq:exchange2}
\begin{aligned}
(1-ab)
\b{C}(b) A(a)
&=
(1-tab)
A(a) \b{C}(b),
\\
(1-ab)
\b{C}(b) B(a)
&=
t(1-ab)
B(a) \b{C}(b)
+
ab (1-t)
A(a) \b{D}(b),
\\
(1-ab)
\b{D}(b) A(a)
&=
(1-ab)
A(a) \b{D}(b)
+
(1-t)
B(a) \b{C}(b),
\\
(1-ab)
\b{D}(b) B(a)
&=
(1-tab)
B(a) \b{D}(b),
\end{aligned}
\end{equation}
where we write
\begin{align*}
A(a) := \lim_{L \rightarrow \infty} A_{L}(a),\quad
B(a) := \lim_{L \rightarrow \infty} B_{L}(a),\quad
\b{C}(b) := \lim_{L \rightarrow \infty} \b{C}_{L}(b),\quad
\b{D}(b) := \lim_{L \rightarrow \infty} \b{D}_{L}(b).
\end{align*}
We recognise on the right hand sides of \eqref{eq:exchange2}, collectively, the Boltzmann weights of the six vertices of the stochastic six vertex model (up to normalization). Graphically, we can write all of these equations in the form
\begin{multline}
\label{graph-exchange}
\left(
\frac{1-a b}{1-t a b}
\right)
\sum_{p_1,p_2,\dots \geq 0}\ \ \
\begin{tikzpicture}[baseline=(current bounding box.center),>=stealth,scale=0.8]
\draw[lgray,ultra thick] (-1,1) node[left,black] {$a$}
-- (4,1) node[right,black] {$j_2$};
\draw[lred,ultra thick] (-1,0) node[left,black] {$b$}
-- (4,0) node[right,black] {$j_1$};
\foreach\x in {0,...,3}{
\draw[lgray,line width=10pt] (3-\x,0.5) -- (3-\x,2);
\draw[lred,line width=10pt] (3-\x,-1) -- (3-\x,0.5);
}
\node[below] at (3,-1) {$m_1$};
\node at (3,0.5) {$p_1$};
\node[above] at (3,2) {$n_1$};
\node[below] at (2,-1) {$m_2$};
\node at (2,0.5) {$p_2$};
\node[above] at (2,2) {$n_2$};
\node[text centered] at (0,0.5) {$\cdots$};
\node[text centered] at (1,0.5) {$\cdots$};
\draw[ultra thick,->] (-1,0) -- (0,0);
\end{tikzpicture}
\quad
=
\\
\quad
\sum_{0 \leq k_1,k_2 \leq 1}\
\sum_{p_1,p_2,\dots \geq 0}\ \ \
\begin{tikzpicture}[baseline=(current bounding box.center),>=stealth,scale=0.8]
\foreach\x in {0,...,3}{
\draw[lgray,line width=10pt] (3-\x,-1) -- (3-\x,0.5);
\draw[lred,line width=10pt] (3-\x,0.5) -- (3-\x,2);
}
\draw[lred,ultra thick] (-1,1) node[left,black] {$b$}
-- (4,1);
\draw[lgray,ultra thick] (-1,0) node[left,black] {$a$}
-- (4,0);
\draw[dotted,thick] (4,1) node[above] {$k_1$} -- (5,0) node[right] {$j_1$};
\draw[dotted,thick] (4,0) node[below] {$k_2$} -- (5,1) node[right] {$j_2$};
\node[below] at (3,-1) {$m_1$};
\node at (3,0.5) {$p_1$};
\node[above] at (3,2) {$n_1$};
\node[below] at (2,-1) {$m_2$};
\node at (2,0.5) {$p_2$};
\node[above] at (2,2) {$n_2$};
\node[text centered] at (0,0.5) {$\cdots$};
\node[text centered] at (1,0.5) {$\cdots$};
\draw[ultra thick,->] (-1,1) -- (0,1);
\end{tikzpicture}
\end{multline}
with the four possibilities in \eqref{eq:exchange2} given by the four possible choices of $j_1,j_2 \in \{0,1\}$.

\subsection{One-variable skew Hall--Littlewood polynomials in terms of row-to-row operators}

At the moment the vectors in $V_1 \otimes V_2 \otimes \cdots$ and $V_1^{*} \otimes V_2^{*} \otimes \cdots$ are parametrized by strings of non-negative integers, \cf Section \ref{sec:row-ops}. It is convenient to extend this notation to partitions, identifying a partition $\mu$ with the string of its multiplicities $(m_1,m_2,\dots)$; $\mu = 1^{m_1} 2^{m_2} \dots$, $m_j = \#\{i : \mu_i = j\}$.

\begin{lemma}
We have the following identities, relating matrix elements of row-to-row operators with one-variable skew Hall--Littlewood polynomials:
\begin{align}
\label{P-one-row}
&
\bra{\lambda} A(a) \ket{\mu}
=
\left( \bm{1}_{\lambda_1' = \mu_1'} \right)
P_{\lambda/\mu}(a),
\quad
&&
\bra{\lambda} B(a) \ket{\mu}
=
\left( \bm{1}_{\lambda_1' = \mu_1'+1} \right)
P_{\lambda/\mu}(a),
\\
\label{Q-one-row}
&
\bra{\mu} \b{C}(b) \ket{\lambda}
=
\left( \bm{1}_{\lambda_1' = \mu_1'+1} \right)
Q_{\lambda/\mu}(b),
\quad
&&
\bra{\mu} \b{D}(b) \ket{\lambda}
=
\left( \bm{1}_{\lambda_1' = \mu_1'} \right)
Q_{\lambda/\mu}(b),
\end{align}
where $\lambda$ and $\mu$ are any two partitions. These relations were first observed in \cite{K}.
\end{lemma}

\begin{proof}
These relations are immediate by comparing the Boltzmann weights \eqref{eq:black-vertices} and \eqref{eq:red-vertices} used in the row-to-row operators with the known expressions for the one-variable skew Hall--Littlewood polynomials:
\begin{align*}
P_{\lambda/\mu}(a)
=
\left\{
\begin{array}{ll}
a^{|\lambda|-|\mu|}
\prod_{i: m_i(\lambda)+1 = m_i(\mu)}
(1-t^{m_i(\mu)}),
&
\quad
\lambda \succ \mu,
\\
0,
&
\quad
{\rm otherwise},
\end{array}
\right.
\end{align*}

\begin{align*}
Q_{\lambda/\mu}(b)
=
\left\{
\begin{array}{ll}
b^{|\lambda|-|\mu|}
\prod_{i: m_i(\lambda) = m_i(\mu)+1}
(1-t^{m_i(\lambda)}),
&
\quad
\lambda \succ \mu,
\\
0,
&
\quad
{\rm otherwise}.
\end{array}
\right.
\end{align*}
The additional constraints on the lengths of the partitions in \eqref{P-one-row}--\eqref{Q-one-row} can be easily deduced by path-conservation.
\end{proof}

\subsection{Writing \eqref{eq:distrib-hl} as an expectation value}

Now we come to the key result of this section, allowing us to express the weights of the probability distribution \eqref{eq:distrib-hl} as expectations in the $t$-boson model.
\begin{theorem}
Let $M,N$ be two positive integers, $S \in \mathcal{S}_{M,N}^{+}$ a binary string, and $\mu \equiv \mu(S)$ the partition associated to $S$. Then
\begin{align}
\label{eq:expectation-skew}
{\rm Prob}^{S}_{M,N}
(
[\la^{(1)}* \dots * \la^{(M+N-1)}]
=
\nu/\mu)
=
\frac{1}{\Pi^{S}(a_1,\dots,a_M;b_1,\dots,b_N)}
\times
\Big\langle 0 \Big|
\prod_{\substack{\longleftarrow \\ i=1 }}^{M+N}
O^{\nu/\mu}_{i}
\Big| 0 \Big\rangle
\end{align}
where we have defined the operators
\begin{align}
\label{eq:skew-ops}
O^{\nu/\mu}_{i}
=
\left\{
\begin{array}{lll}
A(a_{p(i)}),
&
\quad
S(i)=+,
&
\quad
T(i)=+,
\\
B(a_{p(i)}),
&
\quad
S(i)=+,
&
\quad
T(i)=-,
\\
\b{C}(b_{N-m(i)+1}),
&
\quad
S(i)=-,
&
\quad
T(i)=+,
\\
\b{D}(b_{N-m(i)+1}),
&
\quad
S(i)=-,
&
\quad
T(i)=-,
\end{array}
\right.
\end{align}
with $p(i)$ and $m(i)$ denoting the number of pluses and minuses in $(S(1),\dots,S(i))$, and where $T \in \mathcal{S}_{M,N}^{-}$ is the binary string such that $\nu = \nu(T)$.

\begin{proof}
By inserting a complete set of states $\sum_{\lambda^{(i)}} \ket{\lambda^{(i)}} \bra{\lambda^{(i)}}$ between the operators $O^{\nu/\mu}_{i+1}$ and $O^{\nu/\mu}_i$ for all $1 \leq i \leq M+N-1$, and using the single-operator expectation values \eqref{P-one-row}--\eqref{Q-one-row}, it is easy to check that equation \eqref{eq:expectation-skew} transforms into the right hand side of \eqref{eq:distrib-hl}.
\end{proof}

\end{theorem}

It is helpful to think about \eqref{eq:expectation-skew} graphically. Using the graphical representation of the row-to-row operators, \eqref{eq:expectation-skew} can be cast as a partition function on a semi-infinite lattice. The boundary conditions are readily deduced from \eqref{eq:skew-ops}: {\bf 1.} Numbering the rows of the lattice from $1$ to $M+N$, going from top to bottom, the $i$-th external left edge is unoccupied if $S(i) = +$, and occupied if $S(i) = -$. Furthermore, the $i$-th row of the lattice is comprised of vertices \eqref{eq:black-vertices} with parameter $a_{p(i)}$ if $S(i) = +$, and of vertices \eqref{eq:red-vertices} with parameter $b_{N-m(i)+1}$ if $S(i) = -$. {\bf 2.} Similarly, the $i$-th external right edge is unoccupied if $T(i) = +$, and occupied if $T(i) = -$. {\bf 3.} The top and bottom external vertical edges are all unoccupied by paths. See Figure \ref{fig:expectation-example} for an example of this graphical representation.

\begin{figure}
\begin{tikzpicture}[scale=0.8,baseline=(current bounding box.center),>=stealth]
\foreach\y in {0,2,3,5}{
\draw[lgray,thick] (-1,\y) -- (7,\y);
\foreach\x in {0,...,6}{
\draw[lgray,line width=10pt] (\x,\y-0.5) -- (\x,\y+0.5);
}
}
\foreach\y in {-1,1,4}{
\draw[lred,thick] (-1,\y) -- (7,\y);
\foreach\x in {0,...,6}{
\draw[lred,line width=10pt] (\x,\y-0.5) -- (\x,\y+0.5);
}
}
\draw[ultra thick,->] (-1,-1) -- (0,-1); \draw[ultra thick,->] (-1,1) -- (0,1); \draw[ultra thick,->] (-1,4) -- (0,4);
\draw[ultra thick,->] (6,1) -- (7,1); \draw[ultra thick,->] (6,3) -- (7,3); \draw[ultra thick,->] (6,5) -- (7,5);
\node[left] at (-1,-1) {$S(7)=-$}; \node[left] at (-1,0) {$S(6)=+$}; \node[left] at (-1,1) {$S(5)=-$};
\node[left] at (-1,2) {$S(4)=+$};  \node[left] at (-1,3) {$S(3)=+$}; \node[left] at (-1,4) {$S(2)=-$};
\node[left] at (-1,5) {$S(1)=+$};
\node[right] at (7,-1) {$T(7)=+$}; \node[right] at (7,0) {$T(6)=+$}; \node[right] at (7,1) {$T(5)=-$};
\node[right] at (7,2) {$T(4)=+$};  \node[right] at (7,3) {$T(3)=-$}; \node[right] at (7,4) {$T(2)=+$};
\node[right] at (7,5) {$T(1)=-$};
\node[left] at (-3.5,5) {$a_1$};
\node[left] at (-3.5,4) {$b_3$};
\node[left] at (-3.5,3) {$a_2$};
\node[left] at (-3.5,2) {$a_3$};
\node[left] at (-3.5,1) {$b_2$};
\node[left] at (-3.5,0) {$a_4$};
\node[left] at (-3.5,-1) {$b_1$};
\node[below] at (6,-2) {$1$};
\node[below] at (5,-2) {$2$};
\node[below] at (4,-2) {$3$};
\node[below] at (1,-2) {$\longleftarrow$};
\node[below] at (3,-2) {$\cdots$};
\node[below] at (0,-2) {$\infty$};
\end{tikzpicture}
\caption{Graphical representation of the expectation value \eqref{eq:expectation-skew} (suppressing the normalization), in the case $M=4$, $N=3$, $S=(+,-,+,+,-,+,-)$ and $T=(-,+,-,+,-,+,+)$. $N$ paths enter at the left edge of the lattice at the positions of minuses in $S$, and leave via the right edge at the positions of minuses in $T$.}
\label{fig:expectation-example}
\end{figure}
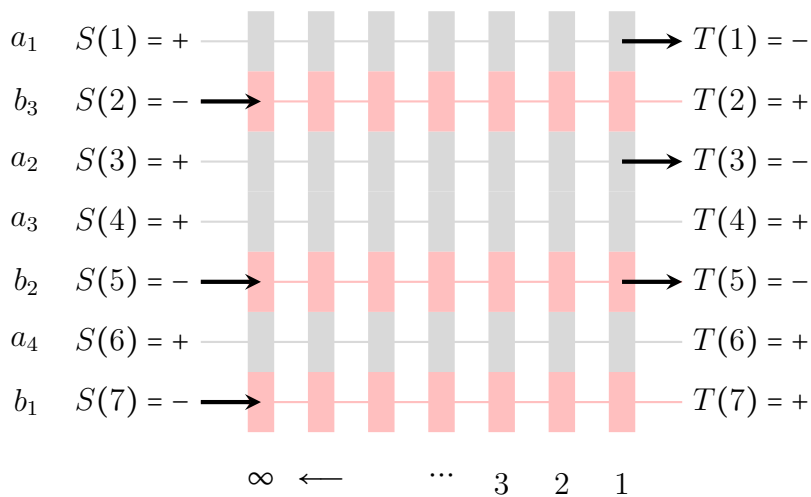

\subsection{Ascending Hall--Littlewood process and stochastic six vertex model on a rectangular domain}

In this section we establish the equivalence between the distributions \eqref{eq:distrib-hl} and \eqref{eq:distrib-6v} in the special case $\mu = \varnothing$. We do this as a warm-up to Theorem \ref{thm:maj}, where the equivalence is proved for arbitrary $\mu$.

\begin{theorem}
\label{thm:min}
Assume that the parameters $(t, \{a_i\}, \{ b_j \} )$ and $( Q, \{\xi_i \}, \{ u_j \} )$ satisfy the conditions \eqref{eq:param-match}, and let $\nu$ be any partition whose Young diagram is contained in the $M \times N$ rectangle. Then we have
\begin{align*}
{\rm Prob}_{M,N}(
[\la^{(1)} \subset \cdots \subset \la^{(M)} \supset \cdots \supset \la^{(M+N-1)}]
=
\nu)
=
{\rm Prob}^{\rm 6V}_{M,N}
(\mathcal{O}(\sigma) = \nu).
\end{align*}
\end{theorem}

\begin{proof}
This is an immediate consequence of the Yang--Baxter equation \eqref{graph-exchange}, which links the two processes. We start from the expression \eqref{eq:expectation-skew} for $S = (+,\dots,+,-,\dots,-)$, and adopt its partition function representation (in what follows we will assume the specific running example $M=4$, $N=3$, $\nu = (3,3,2,1)$, but this is only for definiteness, and the method of proof clearly extends to arbitrary $\nu$):
\begin{multline}
\label{eq:rb-lattice}
{\rm Prob}_{M,N}(
[\la^{(1)} \subset \cdots \subset \la^{(M)} \supset \cdots \supset \la^{(M+N-1)}]
=
\nu)
\\
=
\prod_{i=1}^{M}
\prod_{j=1}^{N}
\left(
\frac{1-a_i b_j}{1-t a_i b_j}
\right)
\times
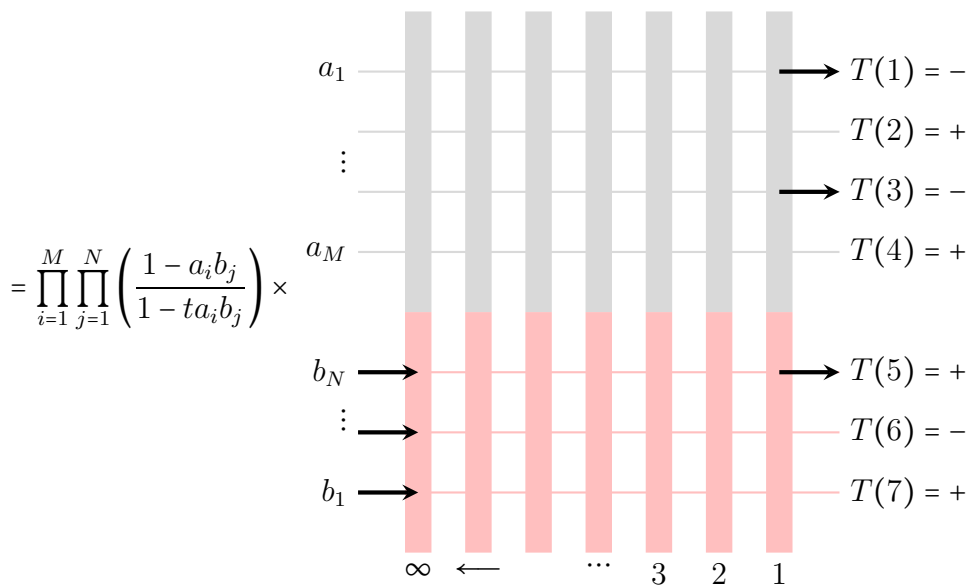
\begin{tikzpicture}[scale=0.8,baseline=(current bounding box.center),>=stealth]
\foreach\x in {0,...,6}{
\draw[lgray,line width=10pt] (\x,2) -- (\x,7);
}
\foreach\y in {3,...,6}{
\draw[lgray,thick] (-1,\y) -- (7,\y);
}
\foreach\x in {0,...,6}{
\draw[lred,line width=10pt] (\x,-2) -- (\x,2);
}
\foreach\y in {-1,...,1}{
\draw[lred,thick] (-1,\y) -- (7,\y);
}
\draw[ultra thick,->] (-1,-1) -- (0,-1); \draw[ultra thick,->] (-1,0) -- (0,0); \draw[ultra thick,->] (-1,1) -- (0,1);
\draw[ultra thick,->] (6,1) -- (7,1); \draw[ultra thick,->] (6,4) -- (7,4); \draw[ultra thick,->] (6,6) -- (7,6);
\node[right] at (7,-1) {$T(7)=+$}; \node[right] at (7,0) {$T(6)=-$}; \node[right] at (7,1) {$T(5)=+$};
\node[right] at (7,3) {$T(4)=+$};
\node[right] at (7,4) {$T(3)=-$};
\node[right] at (7,5) {$T(2)=+$};
\node[right] at (7,6) {$T(1)=-$};
\node[left] at (-1,6) {$a_1$};
\node[left] at (-1,4.5) {$\vdots$};
\node[left] at (-1,3) {$a_M$};
\node[left] at (-1,1) {$b_N$};
\node[left] at (-1,0.2) {$\vdots$};
\node[left] at (-1,-1) {$b_1$};
\node[below] at (6,-2) {$1$};
\node[below] at (5,-2) {$2$};
\node[below] at (4,-2) {$3$};
\node[below] at (1,-2) {$\longleftarrow$};
\node[below] at (3,-2) {$\cdots$};
\node[below] at (0,-2) {$\infty$};
\end{tikzpicture}
\end{multline}
We observe that the multiplicative factor $(1-a_i b_j)/(1-t a_i b_j)$, $1 \leq i \leq M$, $1 \leq j \leq N$, which appears in \eqref{eq:rb-lattice}, is also present on the left hand side of the exchange relation \eqref{graph-exchange}. This suggests that we should use \eqref{graph-exchange} $MN$ times, effectively transferring the red half of the lattice to the top, and the black half to the bottom. The result is
\begin{multline}
{\rm Prob}_{M,N}(
[\la^{(1)} \subset \cdots \subset \la^{(M)} \supset \cdots \supset \la^{(M+N-1)}]
=
\nu)
\\
=
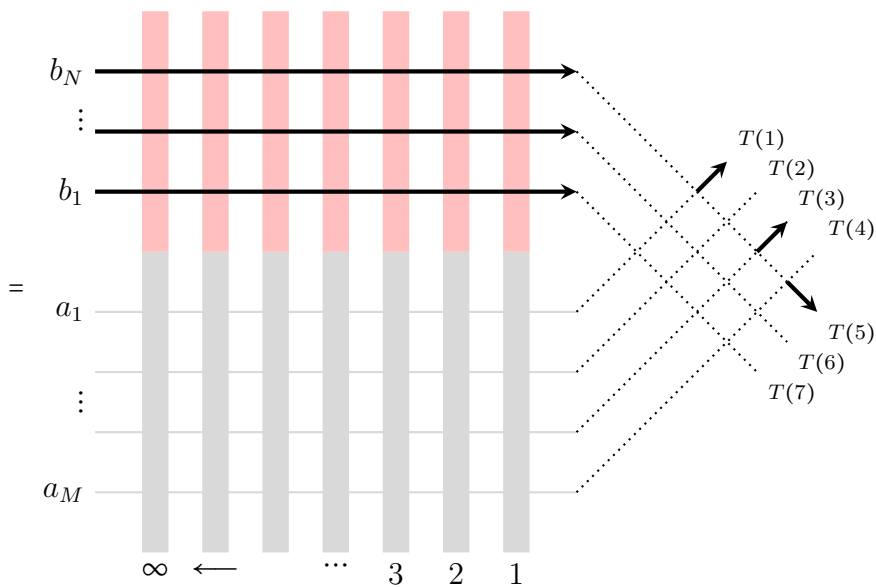
\begin{tikzpicture}[scale=0.8,baseline=(current bounding box.center),>=stealth]
\foreach\x in {0,...,6}{
\draw[lgray,line width=10pt] (\x,-2) -- (\x,3);
}
\foreach\y in {-1,...,2}{
\draw[lgray,thick] (-1,\y) -- (7,\y);
}
\foreach\x in {0,...,6}{
\draw[lred,line width=10pt] (\x,3) -- (\x,7);
}
\foreach\y in {4,...,6}{
\draw[lred,thick] (-1,\y) -- (7,\y);
}
\draw[thick, dotted] (7,6) -- (11,2) node[below right] {\tiny $T(5)$};
\draw[thick, dotted] (7,5) -- (10.5,1.5) node[below right] {\tiny $T(6)$};
\draw[thick, dotted] (7,4) -- (10,1) node[below right] {\tiny $T(7)$};
\draw[ultra thick,->] (10.5,2.5) -- (11,2);
\draw[thick, dotted] (7,2) -- (9.5,4.5) node[above right] {\tiny $T(1)$};
\draw[thick, dotted] (7,1) -- (10,4) node[above right] {\tiny $T(2)$};
\draw[thick, dotted] (7,0) -- (10.5,3.5) node[above right] {\tiny $T(3)$};
\draw[thick, dotted] (7,-1) -- (11,3) node[above right] {\tiny $T(4)$};
\draw[ultra thick,->] (10,3) -- (10.5,3.5); \draw[ultra thick,->] (9,4) -- (9.5,4.5);
\draw[ultra thick,->] (-1,6) -- (7,6); \draw[ultra thick,->] (-1,5) -- (7,5); \draw[ultra thick,->] (-1,4) -- (7,4);
\node[left] at (-1,2) {$a_1$};
\node[left] at (-1,0.5) {$\vdots$};
\node[left] at (-1,-1) {$a_M$};
\node[left] at (-1,6) {$b_N$};
\node[left] at (-1,5.2) {$\vdots$};
\node[left] at (-1,4) {$b_1$};
\node[below] at (6,-2) {$1$};
\node[below] at (5,-2) {$2$};
\node[below] at (4,-2) {$3$};
\node[below] at (1,-2) {$\longleftarrow$};
\node[below] at (3,-2) {$\cdots$};
\node[below] at (0,-2) {$\infty$};
\end{tikzpicture}
\end{multline}
where we have observed that the paths in the red part of the lattice are forced to propagate horizontally to the right edge, while the black part is completely unoccupied. This allows us to suppress the semi-infinite lattice completely (its Boltzmann weight is just 1), and we obtain
\begin{multline}
{\rm Prob}_{M,N}(
[\la^{(1)} \subset \cdots \subset \la^{(M)} \supset \cdots \supset \la^{(M+N-1)}]
=
\nu)
\\
=
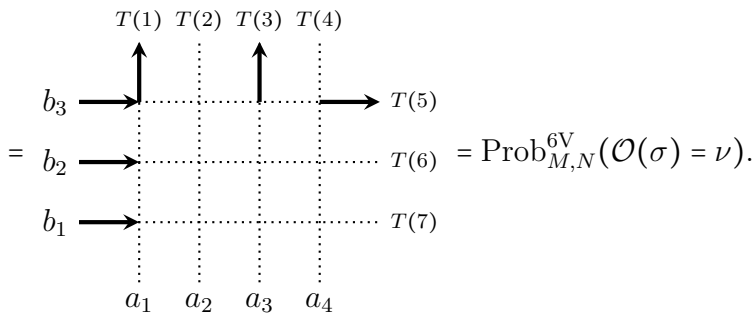
\begin{tikzpicture}[scale=0.8,baseline=(current bounding box.center),>=stealth]
\foreach\x in {1,...,4}{
\draw[thick, dotted] (\x,0) -- (\x,4);
\node[below] at (\x,0) {$a_\x$};
}
\foreach\y in {1,...,3}{
\draw[thick, dotted] (0,\y) -- (5,\y);
\node[left] at (0,\y) {$b_\y$};
\draw[ultra thick,->,rounded corners] (0,\y) -- (1,\y);
}
\foreach\x in {1,3}{
\draw[ultra thick,->,rounded corners] (\x,3) -- (\x,4);
}
\foreach\x in {1,...,4}{
\node[above] at (\x,4) {\tiny $T(\x)$};
}
\foreach\y in {3}{
\draw[ultra thick,->,rounded corners] (4,\y) -- (5,\y);
}
\foreach\y in {5,...,7}{
\node[right] at (5,8-\y) {\tiny $T(\y)$};
}
\end{tikzpicture}
=
{\rm Prob}^{\rm 6V}_{M,N}
(\mathcal{O}(\sigma) = \nu).
\end{multline}

\end{proof}

\subsection{Generic Hall--Littlewood process and stochastic six vertex model on a jagged domain}

\begin{theorem}
\label{thm:maj}
Assume that the parameters $(t, \{a_i\}, \{ b_j \} )$ and $( Q, \{\xi_i \}, \{ u_j \} )$ satisfy the conditions \eqref{eq:param-match}, and let $\nu$, $\mu$ be two partitions whose Young diagrams are contained in the $M \times N$ rectangle, with $\nu \supset \mu$. Let $S \in \mathcal{S}_{M,N}^{+}$ be the length $M+N$ binary string such that $\mu = \mu(S)$. Then
\begin{align*}
{\rm Prob}^{S}_{M,N}(
[\la^{(1)} * \cdots * \la^{(M+N-1)}]
=
\nu/\mu)
=
{\rm Prob}^{\rm 6V}_{M,N}
(\mathcal{O}(\sigma) = \nu/\mu).
\end{align*}
\end{theorem}

\begin{proof}
The proof almost directly repeats that of Theorem \ref{thm:min}, with one exception: in the lattice representation of ${\rm Prob}^{S}_{M,N}([\la^{(1)} * \cdots * \la^{(M+N-1)}] = \nu/\mu)$, the external left edges comprise an arbitrary binary string $S$ of unoccupied/occupied states. Also, the prefactor $1/\Pi^{S}(a_1,\dots,a_M;b_1,\dots,b_N)$ in \eqref{eq:expectation-skew} contains $(1- a_i b_j)/(1-t a_i b_j)$ if and only if the row with parameter $a_i$ is higher than the row with parameter $b_j$. This ensures that $1/\Pi^{S}(a_1,\dots,a_M;b_1,\dots,b_N)$ contains exactly all factors required to transfer the red rows to the top of the lattice, and the black rows to the bottom, by iterating the Yang--Baxter equation \eqref{graph-exchange}. We therefore find that
\begin{multline}
{\rm Prob}^{S}_{M,N}(
[\la^{(1)} * \cdots * \la^{(M+N-1)}] = \nu/\mu)
\\
=
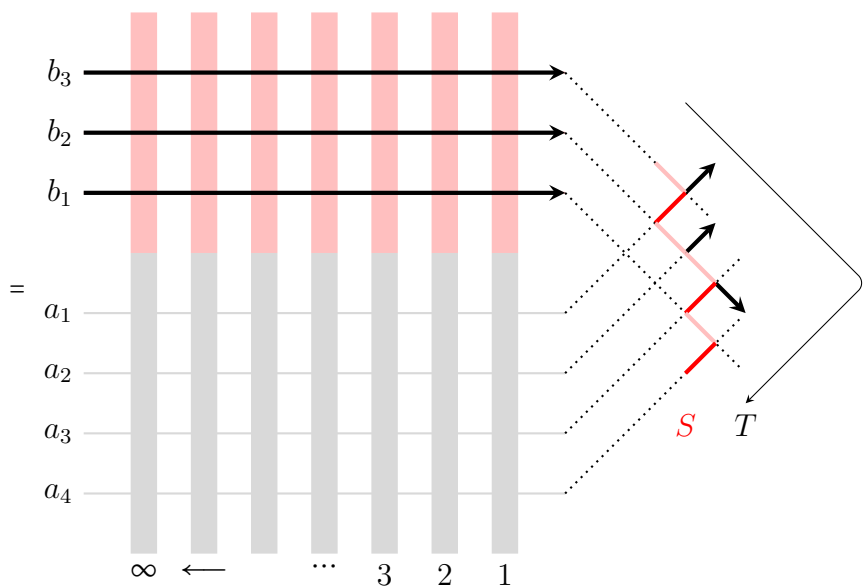
\begin{tikzpicture}[scale=0.8,baseline=(current bounding box.center),>=stealth]
\foreach\x in {0,...,6}{
\draw[lgray,line width=10pt] (\x,-2) -- (\x,3);
}
\foreach\y in {-1,...,2}{
\draw[lgray,thick] (-1,\y) -- (7,\y);
}
\foreach\x in {0,...,6}{
\draw[lred,line width=10pt] (\x,3) -- (\x,7);
}
\foreach\y in {4,...,6}{
\draw[lred,thick] (-1,\y) -- (7,\y);
}
\draw[thick, dotted] (7,6) -- (9.4,3.6);
\draw[thick, dotted] (7,5) -- (9.9,2.1);
\draw[thick, dotted] (7,4) -- (9.9,1.1);
\draw[ultra thick,->] (9.5,2.5) -- (10,2);
\draw[thick, dotted] (7,2) -- (9.4,4.4);
\draw[thick, dotted] (7,1) -- (9.4,3.4);
\draw[thick, dotted] (7,0) -- (9.9,2.9);
\draw[thick, dotted] (7,-1) -- (9.9,1.9);
\draw[ultra thick,->] (9,3) -- (9.5,3.5); \draw[ultra thick,->] (9,4) -- (9.5,4.5);
\draw[ultra thick,lred] (8.5,4.5) -- (9,4);
\draw[ultra thick,red] (9,4) -- (8.5,3.5);
\draw[ultra thick,lred] (8.5,3.5) -- (9,3);
\draw[ultra thick,lred] (9,3) -- (9.5,2.5);
\draw[ultra thick,red] (9.5,2.5) -- (9,2);
\draw[ultra thick,lred] (9,2) -- (9.5,1.5);
\draw[ultra thick,red] (9.5,1.5) -- (9,1);
\node[below] at (9,0.5) {\color{red} $S$};
\draw[thin,rounded corners,->] (9,5.5) -- (12,2.5) -- (10,0.5) node[below] {$T$};
\draw[ultra thick,->] (-1,6) -- (7,6); \draw[ultra thick,->] (-1,5) -- (7,5); \draw[ultra thick,->] (-1,4) -- (7,4);
\node[left] at (-1,2) {$a_1$};
\node[left] at (-1,1) {$a_2$};
\node[left] at (-1,0) {$a_3$};
\node[left] at (-1,-1) {$a_4$};
\node[left] at (-1,6) {$b_3$};
\node[left] at (-1,5) {$b_2$};
\node[left] at (-1,4) {$b_1$};
\node[below] at (6,-2) {$1$};
\node[below] at (5,-2) {$2$};
\node[below] at (4,-2) {$3$};
\node[below] at (1,-2) {$\longleftarrow$};
\node[below] at (3,-2) {$\cdots$};
\node[below] at (0,-2) {$\infty$};
\end{tikzpicture}
\end{multline}
We can then use the same arguments as before to deduce that the semi-infinite lattice appearing above is completely frozen, with weight 1. Deleting it, we are left with a partition function of the six vertex model on a jagged domain (rotated clockwise by 45 degrees). The horizontal lines of this partition function carry the spectral parameters $b_1,\dots,b_N$, numbered from bottom to top. Given that each of its vertices arose from a single application of \eqref{graph-exchange}, we conclude that the length of the $i$-th horizontal line is equal to the number of black rows which sit above the red row with parameter $b_i$ in the starting configuration \eqref{eq:expectation-skew}. Specifying the length of all the horizontal lines specifies the jagged edge uniquely; it is then a simple matter to verify that the binary string determining the jagged edge is exactly $S$ (shown in red above). Meanwhile, the binary string $T$, which encoded the external right edge states in the starting configuration, is transferred directly to the edge states along the jagged domain (preserving the order of $T$). We conclude that
\begin{multline}
{\rm Prob}^{S}_{M,N}(
[\la^{(1)} * \cdots * \la^{(M+N-1)}] = \nu/\mu)
\\
=
\begin{tikzpicture}[scale=0.8,baseline=(current bounding box.center),>=stealth]
\foreach\x in {1,...,4}{
\node[below] at (\x,0) {$a_\x$};
}
\foreach\y in {1,...,3}{
\node[left] at (0,\y) {$b_\y$};
\draw[ultra thick,->,rounded corners] (0,\y) -- (1,\y);
}
\draw[thick, dotted] (1,0) -- (1,3.8); \draw[thick, dotted] (2,0) -- (2,2.8); \draw[thick, dotted] (3,0) -- (3,2.8); \draw[thick, dotted] (4,0) -- (4,1.8);
\draw[thick, dotted] (0,1) -- (4.8,1); \draw[thick, dotted] (0,2) -- (3.8,2); \draw[thick, dotted] (0,3) -- (1.8,3);
\draw[ultra thick,->] (1,3) -- (1,4);
\draw[ultra thick,->] (2,2) -- (2,3);
\draw[ultra thick,->] (3,2) -- (4,2);
\draw[thin,rounded corners,->] (1,4.5) -- (3.5,4.5) -- (3.5,2.5) -- (5.5,2.5) -- (5.5,1) node[below] {$T$};
\end{tikzpicture}
=
{\rm Prob}^{\rm 6V}_{M,N}
(\mathcal{O}(\sigma) = \nu/\mu).
\end{multline}

\end{proof}

\subsection{Return to Theorem \ref{th:equiv-distrib}}

Finally, we wish to translate Theorem \ref{thm:maj} into a statement about the distribution of the lengths of partitions in the Hall--Littlewood process and the distribution of the height function in the stochastic six vertex model.

\begin{proof}[Proof of Theorem \ref{th:equiv-distrib}]
Given a sequence of partitions $\la^{(1)}* \dots * \la^{(M+N-1)}$ arising from the Hall--Littlewood process, the random vector $\{ \lambda'_1(i,N,S) \}_{i=1}^{M+N-1}$ measures the width of the support $[\la^{(1)}* \dots * \la^{(M+N-1)}] = \nu(T)/\mu(S)$ as one walks along the frame of $\mu(S)$. More precisely,
\begin{align*}
\lambda'_1(i,N,S)
=
\frac{1}{2}
\sum_{j=1}^{i}
(S(j)-T(j)),
\quad
\text{for all}\
1 \leq i \leq M+N-1,
\end{align*}
where $T \in \mathcal{S}_{M,N}^{-}$ is the binary string described in Section \ref{sec:support}.

Similarly, given a state $\sigma$ of the six vertex model on the jagged domain obtained by cutting away $\mu(S)$, the random vector $\{y_i(S) - h(x_i(S)+1,y_i(S))\}_{i=1}^{M+N-1}$ measures the difference between the $y$-coordinate of the points $(x_i(S)+1,y_i(S))$ and the value of the height function there. Both can be easily calculated:
\begin{align*}
y_i(S)
=
N-\frac{1}{2}\sum_{j=1}^{i} (1-S(j)),
\quad
h(x_i(S)+1,y_i(S))
=
N-\frac{1}{2}\sum_{j=1}^{i}(1-T(j)),
\end{align*}
where $T$ is the binary string which encodes path occupation along the jagged edge, as explained in Sections \ref{sec:edge-rect} and \ref{sec:jagged}. Therefore,
\begin{align*}
y_i(S)-h(x_i(S)+1,y_i(S))
=
\frac{1}{2}
\sum_{j=1}^{i}
(S(j)-T(j)),
\quad
\text{for all}\
1 \leq i \leq M+N-1,
\end{align*}
and the random vectors $\{ \lambda'_1(i,N,S) \}_{i=1}^{M+N-1}$ and $\{y_i(S) - h(x_i(S)+1,y_i(S))\}_{i=1}^{M+N-1}$ are identically distributed if the random binary string $T$ is identically distributed in the two processes. The latter is the content of Theorem \ref{thm:maj}.

\end{proof}

\section{Hall--Littlewood RSK dynamics}
\label{sec:RSKdyn}

\subsection{Partition dynamics}
\label{sec:HL-dyn-part}

In this section we briefly recall a stochastic dynamics of RSK-type which is related to the Hall--Littlewood polynomials (while the conventional RSK is related to the Schur polynomials). This dynamics was first explicitly obtained in \cite{BufP}; in an implicit form it is present in \cite{BP2}.

Let $c_1, \dots, c_N, \dots$ be positive reals, and let $\mathcal P_1 (c_1), \dots, \mathcal P_N(c_N), \dots$ be a collection of independent Poisson processes in $\mathbb R_{\ge 0}$ with these intensities.

We will consider an evolution of a sequence of random partitions $\la^{(1)} (\tau) \subset \la^{(2)} (\tau) \subset \dots \subset \la^{(N)}(\tau) \subset \cdots$ in time $\tau \in \mathbb R_{\ge 0}$. We assume that for any $k \ge 1$ the partition $\la^{(k)} (\tau) = (\la^{(k)}_1 (\tau) \ge \la^{(k)}_2 (\tau) \ge \dots \ge \la^{(k)}_k (\tau))$ has exactly $k$ rows, some of which may have length 0. Also, let us assume that we have interlacing conditions:
$$
\la_{i+1}^{(m+1)} (\tau) \le \la_i^{(m)} (\tau) \le \la_i^{(m+1)} (\tau), \qquad \mbox{for any $m$ and any $i \le m$}.
$$
In such a sequence, we call a row\footnote{In what follows we do not distinguish between rows and their lengths.} $\la_i^{(m)} (\tau)$ \textit{blocked} if $\la_i^{(m)} (\tau) = \la_{i-1}^{(m-1)} (\tau)$. Otherwise the row is called \textit{free}. Note that for any $m$ the row $\la_1^{(m)} (\tau)$ is always free. We call the row $\la_j^{(m+1)} (\tau)$ the \textit{nearest neighbor} of $\la_i^{(m)} (\tau)$ if it is the smallest free row of $\la^{(m+1)}$ with the condition $j \le i$. See Figure \ref{fig:interlacing_intro} for an example.

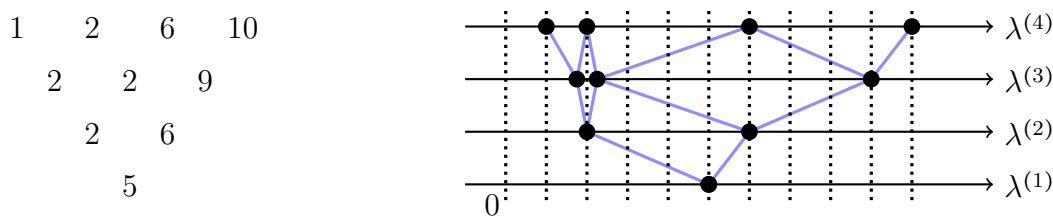
\begin{figure}[htbp]
\begin{center}
\begin{tabular}{rl}
\\ \rule{0pt}{30pt}
\raisebox{7pt}{\begin{tikzpicture}
[scale=1, very thick]
\def\sp{0.09};
\def\x{.5};
\def\y{.7};
\node at (0,0) {5};
\node at (\x,\y) {6};	
\node at (-\x,\y) {2};	
\node at (2*\x,2*\y) {9};	
\node at (0,2*\y) {2};	
\node at (-2*\x,2*\y) {2};	
\node at (-3*\x,3*\y) {1};	
\node at (-1*\x,3*\y) {2};	
\node at (1*\x,3*\y) {6};	
\node at (3*\x,3*\y) {10};	
\end{tikzpicture}}
&\hspace{60pt}
\begin{tikzpicture}
[scale=1, very thick]
\def\sp{0.09};
\def\x{.54};
\def\y{.7};
\def\opac{.45}
\def\wid{1.2}
\draw[line width=\wid, color=blue, opacity=\opac]
(\x,\y)--++(-\x,-\y)--++(-3*\x,\y);
\draw[line width=\wid, color=blue, opacity=\opac]
(4*\x,2*\y)--++(-3*\x,-\y)--++(-4*\x+3/2*\sp,\y)
--++(-3/2*\sp,-\y)--++(-3/2*\sp,\y);
\draw[line width=\wid, color=blue, opacity=\opac]
(5*\x,3*\y)--++(-\x,-\y)--++(-3*\x,\y)
--++(-4*\x+3/2*\sp,-\y)--++(-3/2*\sp,\y)
--++(-3/2*\sp,-\y)--++(-\x+3/2*\sp,\y);
\foreach \pt in
{
(0,0),
(-3*\x,\y), (\x,\y),
(-3*\x-\sp*3/2,2*\y), (-3*\x+\sp*3/2,2*\y), (4*\x,2*\y),
(-4*\x,3*\y), (-3*\x,3*\y), (1*\x,3*\y), (5*\x,3*\y),
}
{
\draw[fill] \pt circle (\sp);
}
\foreach \ll in {1,2,3,4}
{
\draw[->, thick] (-6*\x,\ll*\y-\y) -- (7*\x,\ll*\y-\y)
node[right] {$\la^{(\ll)}$};
}
\foreach \ver in {-5,...,5}
{
\draw[dotted] (\ver*\x,-.3*\y)--(\ver*\x,3.3*\y);
}
\node at (-5*\x-2*\sp,0-3*\sp) {0};
\end{tikzpicture}
\end{tabular}
\end{center}
\caption{A
sequence of partitions $\la^{(1)} \subset \la^{(2)} \subset \la^{(3)} \subset \la^{(4)} = (5) \subset (6,2) \subset (9,2,2) \subset (10,6,2,1)$ and
a visualization of its interlacing conditions. In $\la^{(4)} = (10,6,2,1)$ the rows of lengths $1,6,10$ are free, while the row of length $2$ is blocked. The nearest neighbor of $\la^{(3)}_3 =2$ is $\la^{(4)}_2 =6$. }
\label{fig:interlacing_intro}
\end{figure}

Initially, we set $\la_i^{(m)} (0) = 0$ for all $m$ and $i$. The partitions change only at jump times of the Poisson processes $\{ \mathcal P_i (c_i) \}_{i \ge 1}$ (notation: $\tau \in \mathcal P_i (c_i)$). The probability of the event that some real $\tau$ belongs to two independent Poisson processes simultaneously is zero, so let us assume that $\tau \in \mathcal P_k (c_k)$ for a unique $k$, and describe the evolution of the partitions at this moment.

1) Partitions $\la^{(1)} (\tau), \dots,\la^{(k-1)} (\tau)$ do not change; in each of the partitions $\la^{(k)} (\tau),\la^{(k+1)} (\tau), \dots$ exactly one row increases by 1.

2) In the partition $\la^{(k)} (\tau)$ the smallest free row increases by 1.

3) Modifications propagate to levels $m=k,k+1, \dots,$ as follows:

3a) If the row $\la_i^{(m)} (\tau)$ has increased by one, and $\la_i^{(m)} (\tau) = \la_i^{(m+1)} (\tau)$ before the change, then the row $\la_i^{(m+1)} (\tau)$ increases by one.

3b) Otherwise, if the row $\la_i^{(m)} (\tau)$ has changed, then at the next level either $\la_i^{(m+1)} (\tau)$ or $\la_{i+1}^{(m+1)} (\tau)$ will change. The choice of which one is random; let $R$ and $L$ be corresponding probabilities. These probabilities depend on the array, see below.

4) It remains to determine $R$ and $L$. Let $D$ be the number of rows in $\la^{(m)} (\tau)$ that have the same length as $\la_i^{(m)} (\tau)$ before the change (excluding $\la_i^{(m)} (\tau)$ itself). Then due to the interlacing conditions the number of rows in $\la^{(m+1)} (\tau)$ that are equal to $\la_i^{(m)} (\tau)$ before the change can be either $D$ or $D+1$, see Figure \ref{fig:pushing}.

4a) If it is $D+1$ then $R=1-t$ and $L=t$.

4b) If it is $D$, then $R=\dfrac{1-t}{1-t^{D+1}}$ and $L=1 -R=1-\dfrac{1-t}{1-t^{D+1}}$.

Note that $D$ is allowed to be zero. 

\begin{figure}[htbp]
\begin{center}
	\begin{tabular}{ll}
		\begin{tikzpicture}[
		    scale=1.3,
		    axis/.style={thick, ->, >=stealth'},
		    block/.style ={rectangle, draw=red,
			align=center, rounded corners, minimum height=1em}]
		    \def\y{.7}
		    \draw[axis] (0,0) -- (5,0) node(xline)[right]{$\la^{(m)}$};
		    \draw[axis] (0,\y) -- (5,\y) node(xline)[right]{$\la^{(m+1)}$};
		    \foreach \hh in {.5, 1.5, 2.5, 3.5, 4.5}
		    {
		    	\draw[densely dotted, thick, opacity=.5] (\hh,-.13) -- (\hh,\y+.13);
		    }
		    \def\sp{.13};
		    \def\opac{.25}
			\draw[line width=1.3, color=blue, opacity=\opac]
			(0,2*\y/3)--(.5,\y)--(2.5-\sp,0)--++(\sp/2,\y)--++(\sp/2,-\y)
			--++(\sp/2,\y)--(3.5-\sp/2,0)--++(\sp/2,\y)--++(\sp/2,-\y)
			--(4.5,\y)--(5,3*\y/4);
		    \foreach \pt in
		    {(2.5-\sp,0),(2.5,0),
		    (3.5+\sp/2,0),(3.5-\sp/2,0),
		    (.5,\y),(2.5-\sp/2,\y),(2.5+\sp/2,\y),(3.5,\y),
		    (4.5,\y)}
		    {
		    	\draw[fill] \pt circle (1.4pt);
		    }
		    \draw[->,densely dotted, very thick, color = blue]
	    	(2.5+\sp-.03,-0.03) to [in=180, out=-60] (3,-.4) to [in=-120, out=0] (3.5-\sp/2-.03,-0.05);
	    	\draw[->,densely dotted, ultra thick]
	    	(3.5-\sp/2-.03,-0.05) to [in=-120, out=60] (4.5-.03,\y-0.05);
	    	\draw[->,densely dotted, ultra thick]
	    	(3.5-\sp/2,0) to [in=-80, out=170] (2.5+\sp/2+.02,\y-0.05);
	    	\node at (2.5-\sp,-.26) {$D$};
	    	\node at (2.5,\y+.3) {$D$};
	    	\node at (3.2, -.6) {\scriptsize\color{blue}just increased};
	    	\node at (3.35,.3) {$\la^{(m)}_i$};
	    	\draw (4.3,\y+.5) node[block] (r) {$\frac{1-t}{1-t^{D+1}}$};
	    	\draw (1.4,\y+.5) node[block] (l) {$1-\frac{1-t}{1-t^{D+1}}$};
	    	\draw[color=red] (l.south) -- (2.86,.21);
	    	\draw[color=red] (r.south) -- (4.13,.35);
		\end{tikzpicture}
		&\hspace{10pt}
		\begin{tikzpicture}[
		    scale=1.3,
		    axis/.style={thick, ->, >=stealth'},
		    block/.style ={rectangle, draw=red,
			align=center, rounded corners, minimum height=1em}]
		    \def\y{.7}
		    \draw[axis] (0,0) -- (5,0) node(xline)[right]{$\la^{(m)}$};
		    \draw[axis] (0,\y) -- (5,\y) node(xline)[right]{$\la^{(m+1)}$};
		    \foreach \hh in {.5, 1.5, 2.5, 3.5, 4.5}
		    {
		    	\draw[densely dotted, thick, opacity=.5] (\hh,-.13) -- (\hh,\y+.13);
		    }
		    \def\sp{.13};
		    \def\opac{.25}
			\draw[line width=1.3, color=blue, opacity=\opac]
			(0,\y/4)--(2.5-3*\sp/2,\y)--(2.5-\sp,0)--++(\sp/2,\y)--++(\sp/2,-\y)
			--++(\sp/2,\y)--(3.5-\sp/2,0)--++(\sp/2,\y)--++(\sp/2,-\y)
			--(4.5,\y)--(5,3*\y/4);
		    \foreach \pt in
		    {(2.5-\sp,0),(2.5,0),
		    (3.5+\sp/2,0),(3.5-\sp/2,0),
		    (2.5-3*\sp/2,\y),(2.5-\sp/2,\y),(2.5+\sp/2,\y),(3.5,\y),
		    (4.5,\y)}
		    {
		    	\draw[fill] \pt circle (1.4pt);
		    }
		    \draw[->,densely dotted, very thick, color = blue]
	    	(2.5+\sp-.03,-0.03) to [in=180, out=-60] (3,-.4) to [in=-120, out=0] (3.5-\sp/2-.03,-0.05);
	    	\draw[->,densely dotted, ultra thick]
	    	(3.5-\sp/2-.03,-0.05) to [in=-120, out=60] (4.5-.03,\y-0.05);
	    	\draw[->,densely dotted, ultra thick]
	    	(3.5-\sp/2,0) to [in=-80, out=170] (2.5+\sp/2+.02,\y-0.05);
	    	\node at (2.5-\sp,-.26) {$D$};
	    	\node at (2.5-\sp/2,\y+.3) {$D+1$};
	    	\node at (3.2, -.6) {\scriptsize\color{blue}just increased};
	    	\node at (3.35,.3) {$\la^{(m)}_{i}$};
	    	\draw (4.9,\y+.5) node[block] (r) {$1-t$};
	    	\draw (3.3,\y+.5) node[block] (l) {$t$};
	    	\draw[color=red] (l.south) -- (2.86,.21);
	    	\draw[color=red] (r.west) -- (4.13,.35);
		\end{tikzpicture}
	\end{tabular}
\end{center}
\caption{Pushing and pulling probabilities $R$ and $L$ in the sampling algorithm
(note that $D$ can be zero).}
\label{fig:pushing}
\end{figure}
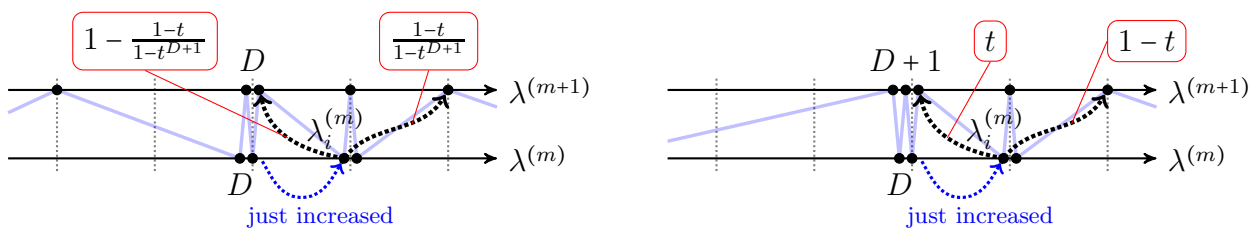

Here is an example of possible changes of the partitions from Figure \ref{fig:interlacing_intro} at a time moment $\tau \in \mathcal P_2 (c_2)$, \ie when a signal comes to the partition $\la^{(2)}$ (boxed numbers indicate rows which have increased).

\begin{align*}
	\framebox{\scalebox{.9}{
	\begin{tikzpicture}
	[scale=.9, thick]
	\def\y{.7}
	\def\t{.27}
	\def\g{11.5}
	\node at (\g,0) {5};
	\node at (\g+\t,\y) {6};\node at (\g-\t,\y) {\framebox{3}};
	\node at (\g+2*\t,2*\y) {9};\node at (\g-0*\t,2*\y) {\framebox{3}};\node at (\g-2*\t,2*\y) {2};
	\node at (\g+3*\t,3*\y) {10};\node at (\g+1*\t,3*\y) {\framebox{7}};\node at (\g-1*\t,3*\y) {2};\node at (\g-3*\t,3*\y) {1};
	\end{tikzpicture}}}
	\raisebox{30pt}{\quad\mbox{with prob. $\dfrac{1-t}{1-t^2}$,}}
	\hspace{50pt}
	\framebox{\scalebox{.9}{
	\begin{tikzpicture}
	[scale=.9, thick]
	\def\y{.7}
	\def\t{.27}
	\def\g{11.5}
	\node at (\g,0) {5};
	\node at (\g+\t,\y) {6};\node at (\g-\t,\y) {\framebox{3}};
	\node at (\g+2*\t,2*\y) {9};\node at (\g-0*\t,2*\y) {\framebox{3}};\node at (\g-2*\t,2*\y) {2};
	\node at (\g+3*\t,3*\y) {10};\node at (\g+1*\t,3*\y) {6};\node at (\g-1*\t,3*\y) {\framebox{3}};\node at (\g-3*\t,3*\y) {1};
	\end{tikzpicture}}}
	\raisebox{30pt}{\quad\mbox{with prob. $1 - \dfrac{1-t}{1-t^2}$.}}
\end{align*}

Let us explain how this dynamics is related to the Hall--Littlewood process. In this section the parameters $\{ c_i \}$ play the role of parameters $\{ a_i \}$ in the definitions of Sections \ref{sec:general-HL} and \ref{sec:HL-process-usual}.

\begin{proposition}
\label{prop:dynamics-1}
For any $\tau>0$ and any positive integers $n_1, \dots, n_k$, the partitions $\lambda^{(n_1)} (\tau), \dots, \la^{(n_k)} (\tau)$ are distributed as the correspondingly numbered partitions of the ascending Hall--Littlewood process determined by parameters $\{ c_i \}_{i \ge 1}$ and the Plancherel specialization with parameter $\tau$ (see Section \ref{sec:HL-process-usual}).
\end{proposition}
\begin{proof}
This is \cite[Theorem 1.10]{BufP}.
\end{proof}

For positive integers $M,N$ let $S \in \mathcal{S}_{M,N}^{+}$ be a string of signs as defined in Section \ref{sec:general-HL}, and let $\hat \tau_1 \le \hat \tau_2 \le \dots \le \hat \tau_N$ be positive reals. We write ${\rm next} (\hat \tau_i)$ for $\hat \tau_{i+1}$.
Analogously to the construction in Section \ref{sec:more-gen-th}, let $\{ n_i (S) \}_{i=0}^{M+N}$, $\{ \tau_i (S) \}_{i=1}^{M+N}$ be collections such that

\textbf{1.} $n_0 (S)=M$, $n_{M+N} (S)=0$, $\tau_1 (S) = \hat \tau_1$, $\tau_{M+N} (S) = \hat \tau_N$. 

\textbf{2.} If $S(i)=+$, then $n_i (S)= n_{i-1} (S)$, $\tau_i (S) = {\rm next}(\tau_{i-1} (S))$. If $S(i)=-$, then $n_i (S)= n_{i-1} (S)-1$, $\tau_i (S) = \tau_{i-1} (S)$.

\begin{proposition}
\label{prop:dynamics-staircase}
For any positive integers $M,N$, any reals $0< \hat \tau_1 \le \hat \tau_2 \le \dots \le \hat \tau_N$, and any string $S \in \mathcal{S}_{M,N}^{+}$, the partitions $\lambda^{(n_1 (S) )} (\tau_1 (S) ), \dots, \la^{(n_{M+N-1} (S)) } (\tau_{M+N-1} (S))$ are distributed as the correspondingly numbered partitions of the Hall--Littlewood process determined by the string $S$, parameters $\{ c_i \}_{i \ge 1}$, $\{ \hat \tau_i \}$, and the Plancherel specializations (see the end of Section \ref{sec:general-HL}).
\end{proposition}
\begin{proof}

This claim follows from the general machinery of RSK-type dynamics for Macdonald processes developed in \cite{BP2}. Let us provide details of this derivation.

For partitions $\bar \lambda \subset \la$ define
$$
\Lambda_k^{k+1} ( \bar \la, \la) := \frac{ P_{\bar \la} (c_1, \dots, c_k)}{ P_{\la} (c_1, \dots, c_{k+1})} P_{\la / \bar \la} (c_{k+1}).
$$
Let $\mathcal P_N$ be a collection of partitions $\mu^{(1)} \subset \mu^{(2)} \subset \dots \subset \mu^{(N)}$ such that each $\mu^{(k)}$ has no more than $k$ rows. A probability measure $\mathcal{M}$ on $\mathcal P_N$ is called a \textit{Gibbs measure} (with respect to the Hall--Littlewood process with parameters $c_1, \dots, c_N$) if
$$
\mathcal{M} (\mu^{(1)} \subset \mu^{(2)} \subset \dots \subset \mu^{(N)}) = \mathcal{M}_{(N)} (\mu^{(N)}) \prod_{k=1}^{N-1} \Lambda_k^{k+1} ( \mu^{(k)}, \mu^{(k+1)}),
$$
where $\mathcal{M}_{(N)}$ is the projection of $\mathcal{M}$ onto $\mu^{(N)}$.

For partitions $\bar \lambda \subset \la$ consisting of no more than $k$ rows and such that $\la / \bar \la$ has exactly one box let
$$
Q_k (\bar \la \to \la) := \frac{P_{\la} (c_1, \dots, c_k)}{P_{\bar \la} (c_1, \dots, c_k)} \psi'_{\bar \la / \la},
$$
where $\psi'_{\bar \la / \la}$ is a constant given by \cite[Chapter 6, Equation 6.24.iv]{M}. A \textit{univariate dynamics} on partitions with no more than $k$ rows is a continuous time dynamics which has jump rate probabilities $Q_k$. A key property of this dynamics is that the transition probability over a time period of length $\tilde \tau>0$ has a form
$$
\mathbf P_{\tilde \tau} ( \bar \la \to \la) := Q_{\la / \bar \la} ({\rm Pl}_{\tilde \tau}) \frac{P_{\la} (c_1, \dots, c_k)}{P_{\bar \la} (c_1, \dots, c_k)},
$$
where ${\rm Pl}_{\tilde \tau}$ is the Plancherel specialization with parameter $\tilde\tau$.

Now we are able to formulate three essential properties of this dynamics:

1) For any $\tau$ and $k$ the distribution of $\la^{(1)} (\tau) \subset \dots \subset \la^{(k)} (\tau)$ is Gibbs.

2) For any $k$ the restriction of our dynamics to $\la^{(k)} (\tau)$ is univariate (with parameters $c_1, \dots, c_k$).

3) The transition probability $\mathbf P ( \la^{(k)} (\tau) \to \la^{(k)} (\tau + \tilde \tau))$ does not depend on partitions $\la^{(l)}$ with $l>k$.

These claims are established in \cite[Proposition 5.3]{BP2} (see also \cite[Theorem 4.6]{BufP}). Also, the third property directly follows from the definition of the dynamics.

We proceed to the proof of the proposition. To simplify notations, let us consider the case of $M=N$ and $S=(+,-,+,-, \dots, +,-)$ first. For such a string we have $\tau_1 (S) = \hat \tau_1$, ..., $\tau_N (S) = \hat \tau_N$; we omit the dependence of $\tau$ on $S$ in notations.

By Proposition \ref{prop:dynamics-1} the probability of the event $(\la^{(N)} (\tau_1) = \mu^{(N)}, \la^{(N-1)} (\tau_1) = \tilde \mu^{(N-1)} )$ is proportional to
$$
Q_{\mu^{(N)}} ( \mathrm{Pl}_{\tau_1}) P_{\mu^{(N)} / \tilde \mu^{(N-1)}} (a_N) P_{\tilde \mu^{(N-1)}} (c_1, \dots, c_{N-1}).
$$
By property 2) the evolution on the level $(N-1)$ from the moment $\tau_1$ till the moment $\tau_2$ is univariate. Using this and property 3), we deduce that the probability of the event $(\la^{(N)} (\tau_1) = \mu^{(N)}, \la^{(N-1)} (\tau_1) = \tilde \mu^{(N-1)}, \la^{(N-1)} (\tau_2) = \mu^{(N-1)})$ is proportional to
\begin{multline}
\label{eq:univ-dyn-step}
Q_{\mu^{(N)}} ( \mathrm{Pl}_{\tau_1}) P_{\mu^{(N)} / \tilde \mu^{(N-1)}} (c_N) P_{\tilde \mu^{(N-1)}} (c_1, \dots, c_{N-1}) \mathbf P_{\tau_2-\tau_1} ( \tilde \mu^{(N-1)} \to \mu^{(N-1)}) \\ = Q_{\mu^{(N)}} ( \mathrm{Pl}_{\tau_1}) P_{\mu^{(N)} / \tilde \mu^{(N-1)}} (c_N) Q_{\mu^{(N-1)} / \tilde \mu^{(N-1)} } (\mathrm{Pl}_{\tau_2 - \tau_1}) P_{\mu^{(N-1)}} (c_1, \dots, c_{N-1}).
\end{multline}
Applying Gibbs property 1) for time $\tau_2$, we obtain that the probability of the event $(\la^{(N)} (\tau_1) = \mu^{(N)}, \la^{(N-1)} (\tau_1) = \tilde \mu^{(N-1)}, \la^{(N-1)} (\tau_2) = \mu^{(N-1)}, \la^{(N-2)} (\tau_2) = \tilde \mu^{(N-2)} )$ is proportional to
\begin{multline}
\label{eq:gibbs-step}
Q_{\mu^{(N)}} ( \mathrm{Pl}_{\tau_1}) P_{\mu^{(N)} / \tilde \mu^{(N-1)}} (c_N) Q_{\mu^{(N-1)} / \tilde \mu^{(N-1)} } (\mathrm{Pl}_{\tau_2 - \tau_1}) P_{\mu^{(N-1)}} (c_1, \dots, c_{N-1}) \Lambda_{N-2}^{N-1} ( \tilde \mu^{(N-2)} , \mu^{(N-1)} ) \\ = Q_{\mu^{(N)}} ( \mathrm{Pl}_{\tau_1}) P_{\mu^{(N)} / \tilde \mu^{(N-1)}} (c_N) Q_{\mu^{(N-1)} / \tilde \mu^{(N-1)} } (\mathrm{Pl}_{\tau_2 - \tau_1}) P_{\mu^{(N-1)} / \tilde \mu^{(N-2)}} (c_{N-1}) P_{\tilde \mu^{(N-2)}} (c_1, \dots, c_{N-2}).
\end{multline}
Iterating steps of the form \eqref{eq:univ-dyn-step} and \eqref{eq:gibbs-step} for all levels, we arrive at the statement of the proposition. 
For a general string $S$, one applies the step \eqref{eq:univ-dyn-step} for each plus in the string, and the step \eqref{eq:gibbs-step} for each minus in $S$; this yields the definition of the required Hall--Littlewood process very similarly to what was done above.

\end{proof}

\subsection{Dynamics on sets}
\label{sec:dyn-sets}

Let us give an equivalent form of the dynamics above, which will bring us closer to the combinatorics of vertex models.

Consider sets $\{ V_i (\tau) \}_{i \in \mathbb Z_{\ge 1}; \tau \in \mathbb R_{\ge 0}}$; each $V_i (\tau)$ is a subset of $\mathbb Z_{\ge 0}$. Let $h^{(r)} (\tau;m)$ be the number of sets from $V_1 (\tau), \dots, V_m (\tau)$ that contain the element $r$. Obviously, the sets $\{ V_i (\tau) \}_{i \in \mathbb Z_{\ge 1}}$ and $\{ h^{(r)} (\tau;m) \}_{m \ge 1; r \ge 0}$ determine each other for any fixed $\tau$.

Initially, we set $V_i (0) = \mathbb Z_{\ge 0} $ for all $i$. The sets change only at times which appear in the Poisson processes $\{ \mathcal P_i (c_i) \}_{i \ge 1}$. Again, we consider the case when $\tau \in \mathcal P_k (c_k)$ for a unique $k$ only, and describe the evolution of sets at this moment.

1) Sets $V_{1} (\tau), \dots,V_{k-1} (\tau)$ do not change, in the set $V_{k} (\tau)$ one element is removed, and for each of the sets $V_{k+1} (\tau), V_{k+2} (\tau), \dots$ either this set does not change or exactly one element is added and exactly one element is removed in it.

2) In the set $V_{k} (\tau)$, it is its smallest element that is being removed.

3) Let us describe the change of $V_m (\tau)$ for $m=k+1, k+2, \dots$:

Let $m_{-}$ be the largest integer such that $m_{-} < m$ and $V_{m_{-}} (\tau)$ has changed. Then the sets $V_{m_{-} +1} (\tau), \dots, V_{m -1} (\tau)$ did not change. Let $i$ be the element that was removed from $V_{m_{-}} (\tau)$.

3a) If $i$ is contained in $V_m (\tau)$, then $V_m (\tau)$ does not change.

3b) Otherwise, either $i$ is added to $V_m (\tau)$ and the smallest element $>i$ of $V_m (\tau)$ is removed from it, or $V_m (\tau)$ does not change. The choice of which operation of these two happens is random; let $R$ and $L$ be these probabilities, respectively. The probabilities depend on the sets.

4) It remains to determine $R$ and $L$.

4a) If $V_m (\tau)$ does not contain $i-1$, then $R=1-t$ and $L=t$.

4b) If $V_m (\tau)$ contains $i-1$, then $$R=\frac{1-t}{1-t^{h^{(i)} (\tau-0;m) - h^{(i-1)} (\tau-0;m) +1}}, \qquad L = 1 - \frac{1-t}{1-t^{h^{(i)} (\tau-0;m) - h^{(i-1)} (\tau-0;m) +1}},$$
where $h^{(i)} (\tau-0;m) $ are computed before the changes related to the arrival of a Poisson signal at time $\tau$ happen (to indicate this, we write $\tau-0$ in the formula). Note that the rules of the dynamics imply that $h^{(i)} (\tau;m) - h^{(i-1)} (\tau;m)$ is always nonnegative.

See Figure \ref{fig:set-dyn} for an example.

\begin{figure}
\begin{align*}
\begin{tikzpicture}[>=stealth,scale=0.8]
\node at (0,0) {0,1,2};
\node at (0,1) {0,1,2};
\node at (0,2) {0,1,2};
\draw[thick,->] (0.7,1) -- (2.2,1); \node at (1.5,1.4) {1};
\node at (3,0) {0,1,2};
\node at (3,1) {0,1,2};
\node at (3,2) {1,2};
\draw[thick,->] (3.7,1) -- (5.2,1); \node at (4.5,1.4) {1-$t$};
\node at (6,0) {1,2};
\node at (6,1) {0,1,2};
\node at (6,2) {0,2};
\draw[thick,->] (6.7,1) -- (8.2,1); \node at (7.5,1.4) {1};
\node at (9,0) {1,2};
\node at (9,1) {1,2};
\node at (9,2) {0,2};
\draw[thick,->] (9.7,1) -- (11.2,1); \node at (10.5,1.4) {1 - $\frac{1-t}{1-t^2}$};
\node at (12,0) {1,2};
\node at (12,1) {2};
\node at (12,2) {0,2};
\end{tikzpicture}
\end{align*}
\caption{An example of dynamics on sets in which independent signals arrive to levels $3,1,2,2$ and the probabilities of drawn changes are indicated. Only elements $0,1,2$ are written (all other elements do not change; in fact, element $2$ also does not change in this example). }
\label{fig:set-dyn}
\end{figure}
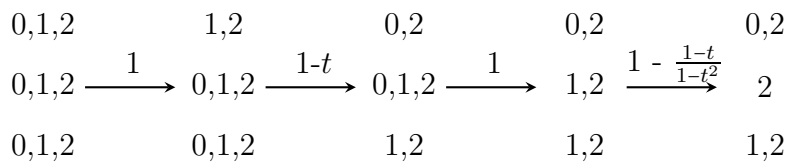

\begin{figure}
\begin{tikzpicture}[>=stealth,scale=0.8]

\draw[ultra thick,->] (0,0) -- (1.7,0);
\draw[ultra thick,->, draw=red,fill=red] (0,0.2) -- (6,0.2);
\draw[ultra thick,->] (0,1) -- (2.5,1);
\draw[ultra thick,->, draw=red,fill=red] (0,1.2) -- (4,1.2);
\draw[ultra thick, draw=red] (4,1.2) -- (4,3);
\draw[ultra thick,->] (0,2) -- (1,2);
\draw[ultra thick,->, draw=red,fill=red] (0,2.2) -- (1.7,2.2);
\draw[ultra thick, draw=red] (1.7,2.2) -- (1.7,3);
\draw[ultra thick] (1,2) -- (1,3);
\draw[ultra thick,->] (1.7,0) -- (1.7,2);
\draw[ultra thick,->] (1.7,2) -- (6,2);
\draw[ultra thick] (2.5,1) -- (2.5,3);
\end{tikzpicture}

\caption{An arrow model for an example from Figure \ref{fig:set-dyn}. Black arrows indicate the presence of element $0$, red elements indicate the presence of element $1$. }
\label{fig:colored-arrows}
\end{figure}
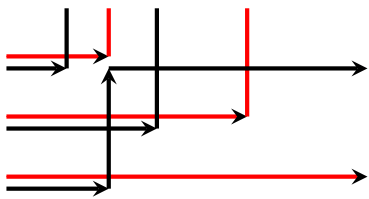

For each $m, \tau$ consider the partition $\tilde \la^{(m)} (\tau):= (m - h^{(0)} (\tau,m), m - h^{(1)} (\tau,m), \dots)$.

\begin{proposition}
\label{prop:dynamics-2}
Random time-dependent sequences of partitions $\{ \tilde \lambda^{(m)} (\tau) \}_{m \in \mathbb Z_{\ge 0}, \tau \in \mathbb R_+}$ and $\{ \lambda^{(m)} (\tau) \}_{m \in \mathbb Z_{\ge 0}, \tau \in \mathbb R_+}$ are identically distributed (the latter one was constructed in Section \ref{sec:HL-dyn-part}).

\end{proposition}
\begin{proof}
Let us say that the element $i$ belongs to the set $V_m (\tau)$ if and only if the number of rows $\le i$ in $\la^{(m+1)} (\tau)$ and the number of rows $\le i$ in $\la^{(m)} (\tau)$ differ exactly by 1 (note that interlacing conditions for partitions imply that these numbers are either equal or differ by one). Then the dynamics on sets coincides with the dynamics described in Section \ref{sec:HL-dyn-part}; all rules correspond to each other.
\end{proof}

Let us consider the restriction of this set dynamics obtained by looking only at the elements $\{0,1, \dots, k\}$ in the $V_i (\tau)$'s. From the evolution rules, it is clear that this restricted dynamics is Markovian (indeed, at any moment the evolution of the element $i$ depends only on itself and the positions of elements $i-1$).
In particular, the evolution of the elements $0$ is Markovian; note that in this dynamics the case 4b) cannot happen.

It is possible to represent the dynamics on the sets by drawing vertex model-type arrows in the following way. For $\tau_1 < \tau_2$ we have a horizontal arrow of type $i$ between $(\tau_1, n)$ and $(\tau_2,n)$ if $i \in V_n (\tau)$ for any $\tau \in (\tau_1, \tau_2)$. For $n_1 < n_2$ we draw a vertical arrow of type $i$ between $(\tau,n_1)$ and $(\tau,n_2)$ if at the moment $\tau$ element $i$ was removed from $V_{n_1} (\tau)$ and was added to $V_{n_2} (\tau)$; we also draw an infinite vertical line of type $i$ starting from $(\tau,n_1)$ if an element $i$ was removed from $V_{n_1} (\tau)$ and was not added elsewhere. See Figure \ref{fig:colored-arrows} for an example.

The dynamics of lines of type $0$ coincides with a version of the stochastic six vertex model. We discuss this in Section \ref{sec:half-cont}.
Also, one can consider an interacting particle system given by the Markov dynamics on the $0$ elements\footnote{In a similar way, the restriction of the dynamics to the elements $\{0,\dots, k\}$ leads to a certain multi-species particle dynamics.}.
We call it the \textit{t-PushTASEP}; let us describe it separately.

Consider $\mathbb Z_{\ge 1}$, and let each site $l$ have an exponential clock with rate $c_l$. When the clock at site $k$ rings, then there are two variants:

a) The site $k$ is empty --- then nothing happens.

b) The site $k$ contains a particle --- then this particle becomes active.

The evolution of an active particle is further described by the following rules:

c) The active particle always jumps to the right by $1$.

c1) If this site is occupied, then the active particle stops, but it pushes the particle in this site, which becomes active instead of the stopped one.

c2) If the site is empty, then the active particle can either stop with probability $1-t$ (and be no longer active), or continue to be active with probability $t$.


If we start this dynamics with the configuration in which each positive integer $k$ is occupied by a particle, then Proposition \ref{prop:dynamics-2} allows to relate the evolution to the first column of partitions in a Hall--Littlewood process. One can argue that the t-PushTASEP is then a novel example of an \textit{integrable} interacting particle system.

\subsection{Half-continuous stochastic six vertex model}
\label{sec:half-cont}

In parameters of the Hall--Littlewood process $\{a_x\}_{x \in \mathbb Z_{\ge 1}}$, $\{b_y\}_{y \in \mathbb Z_{\ge 1}}$ (see Section \ref{sec:HL-process-usual}), the weights of stochastic six vertex model have the form \eqref{eq:six-vertex-weights}.
Let us set all $a_i$ to be equal to $\eps / (1-t)$, and let the index $x$ range over the grid $\eps \mathbb Z_{\ge 1}$ rather than $\mathbb Z_{\ge 1}$. Consider now the $\eps \to 0$ limit. The grid in the $x$-direction becomes continuous; however, up-right paths still make sense. In such a
limit we obtain $\mathrm{Prob} (\uu) = t$, $\mathrm{Prob} (\ur) = 1-t$, while the weight $\mathrm{Prob} (\ru)$ turns into the infinitesimal jumping probability of the Poisson process of intensity $b_y$. We refer to the limiting model as to the \textit{half-continuous} stochastic six vertex model; let $\{h(\tau, y)\}_{\tau \in \mathbb R_{\ge 0}, y \in \mathbb Z_{\ge 1}}$ be its height function.

Denote by $l (\tau,y)$ the length of the first column of $\lambda^{(y)} (\tau)$, where $\lambda^{(y)} (\tau)$ are the random partitions introduced in Section \ref{sec:HL-dyn-part}.

\begin{theorem}
\label{th:1-2cont6vert}
In the notations above and with identification $b_i = c_i$ for all $i$, the random fields $\{ h(\tau, y) \}_{\tau \in \mathbb R_{\ge 0}, y \in \mathbb Z_{\ge 1}}$ of the half-continuous stochastic six vertex model, and $\{ (y - l(\tau,y)) \}_{\tau \in \mathbb R_{\ge 0}, y \in \mathbb Z_{\ge 1}}$ are identically distributed.

\end{theorem}
\begin{proof}
One readily sees that the up-right paths of the half-continuous stochastic six vertex model exactly match the vertex-type lines of dynamics of the element $0$ (see Section \ref{sec:dyn-sets}). Thus, the statement follows from Proposition \ref{prop:dynamics-2} applied to the first columns of the corresponding partitions.
\end{proof}

\begin{corollary}
\label{cor:last}
The distribution of the height function along the down-right paths coincide with the Hall--Littlewood processes of Proposition \ref{prop:dynamics-staircase}.
\end{corollary}
\begin{remark}
Corollary \ref{cor:last} follows either from Theorem \ref{th:1-2cont6vert} or from Theorem \ref{th:equiv-distrib} (in the limit of parameters described above); this yields two independent proofs.
\end{remark}

\end{document}